\documentclass[11pt]{amsart}
\usepackage{amssymb,amscd}

\allowdisplaybreaks
\DeclareMathOperator{\Spec}{Spec}

\begin{document}
\newtheorem{theorem}{Theorem}[section]
\newtheorem{lemma}[theorem]{Lemma}
\newtheorem{proposition}[theorem]{Proposition}
\newtheorem{corollary}[theorem]{Corollary}
\newtheorem{example}[theorem]{Example}
\newtheorem{remark}[theorem]{Remark}
\newtheorem{prob}[theorem]{Problem}
\newtheorem{theoremA}{Theorem}
\renewcommand{\thetheoremA}{}
\theoremstyle{definition}
\newtheorem{defi}[theorem]{Definition}
\renewcommand{\thedefi}{}
\newtheorem{nota}[theorem]{Notation}
\renewcommand{\thedefi}{}
\input amssym.def
\long\def\alert#1{\smallskip{\hskip\parindent\vrule%
\vbox{\advance\hsize-2\parindent\hrule\smallskip\parindent.4\parindent%
\narrower\noindent#1\smallskip\hrule}\vrule\hfill}\smallskip}
\def\ff{\frak}
\def\Spec{\mbox{\rm Spec}}
\def\type{\mbox{ type}}
\def\Hom{\mbox{ Hom}}
\def\rk{\rm{rk}}
\def\Ext{\mbox{ Ext}}
\def\Ker{\mbox{ Ker}}
\def\Max{\mbox{\rm Max}}
\def\Min{\mbox{\rm Min}}
\def\Maxi{\mbox{\rm Max}^{-1}}
\def\End{\mbox{\rm End}}
\def\Inv{{\rm Inv}}
\def\Div{{\rm Div}}
\def\Uni{{\rm Uni}}
\def\l{\langle\:}
\def\r{\:\rangle}
\def\Rad{\mbox{\rm Rad}}
\def\Zar{\mbox{\rm Zar}}
\def\Supp{\mbox{\rm Supp}}
\def\Rep{\mbox{\rm Rep}}
\def\cal{\mathcal}
\def\O{{\cal O}}
\def\L{{\cal L}}
\def\F{{\cal F}}
\def\Ms{{\cal M}_{\#}}
\def\Md{{\cal M}_{\dagger}}

\newcommand{\Z}{\mathbb{Z}}
\newcommand{\N}{\mathbb{N}}

\title[Radical factorization in finitary ideal systems]{Radical factorization in finitary ideal systems}

\author{Bruce Olberding}
\address{Department of Mathematical Sciences, New Mexico State University, Las Cruces, NM 88003, United States}
\email{bruce@nmsu.edu}

\author{Andreas Reinhart}
\address{Department of Mathematical Sciences, New Mexico State University, Las Cruces, NM 88003, United States}
\email{andreas.reinhart@uni-graz.at}


\begin{abstract}
In this paper we investigate the concept of radical factorization with respect to finitary ideal systems of cancellative monoids. We present new characterizations for $r$-almost Dedekind $r$-SP-monoids and provide specific descriptions of $t$-almost Dedekind $t$-SP-monoids and $w$-SP-monoids. We show that a monoid is a $w$-SP-monoid if and only if the radical of every nontrivial principal ideal is $t$-invertible. We characterize when the monoid ring is a $w$-SP-domain and describe when the $\ast$-Nagata ring is an SP-domain for a star operation $\ast$ of finite type.
\end{abstract}

\maketitle

\section{Introduction}

The concept of factoring ideals into radical ideals has been studied by various authors. It started with papers by Vaughan and Yeagy \cite{V,Y} who studied radical factorization in integral domains. They showed that every integral domain for which every ideal is a finite product of radical ideals (we will call such a domain an SP-domain) is an almost Dedekind domain. Later the first-named author gave a complete characterization in \cite{Olb} of SP-domains in the context of almost Dedekind domains. After that the second-named author investigated the concept of radical factorization in \cite{R,RR} with respect to finitary ideal systems. Further progress in describing SP-domains was made in \cite{FHL, HOR, MC}. Besides that, radical factorization in commutative rings with identity was investigated in \cite{AD}. Many of these results were extended in a recent paper \cite{OR} where radical factorization was studied in the context of principally generated $C$-lattice domains.

Ideal systems of monoids are a generalization of star operations of integral domains. They were studied in detail in \cite{HA}. It turns out that (finitary) ideal systems in general fail to be modular (i.e., the lattice of ideals induced by the ideal system is not modular). In particular, an $r$-SP-monoid (i.e., the ``ideal system theoretic analogue'' of an SP-domain) can fail to be $r$-almost Dedekind. The goals of this paper are manyfold. We extend the known characterizations of $r$-almost Dedekind $r$-SP-monoids for finitary ideal systems $r$. We consider lattices of ideals that are (a priori) neither principally generated nor modular. Thus we complement the results of \cite{OR} by describing the lattice of $r$-ideals in case that $r$ is a (not necessarily modular) finitary ideal system. Let $p$ be a modular finitary ideal system and $r$ a finitary ideal system such that every $r$-ideal is a $p$-ideal. Then there is a modular finitary ideal system $\widetilde{r_p}$ ``between'' $r$ and $p$, called the $p$-modularization of $r$, which can be used to describe the $r$-ideals. We show, for instance, that a monoid is an $\widetilde{r_p}$-SP-monoid if and only if every minimal prime $s$-ideal of a nontrivial $r$-finitely generated $r$-ideal is of height one and the radical of every nontrivial principal ideal is $r$-invertible. We put particular emphasis on the $t$-system and its modularizations (like the $w$-system) and present stronger characterizations for these types of ideal systems. As an application we investigate several ring-theoretical constructions with respect to the aforementioned properties.

In Section 2 we introduce the notion of (finitary) ideal systems and most of the important terminology. We also show the most basic properties of the modularizations of a finitary ideal system. In Section 3 we study finitary ideal systems in general. Our main results are characterization theorems for $r$-almost Dedekind $r$-SP-monoids as well as $r$-B\'ezout $r$-SP-monoids.

We put our focus on the $t$-system and its modularizations in Section 4. We will show that a monoid is a $w$-SP-monoid if and only if the radical of every nontrivial principal ideal is $t$-invertible. Moreover, we show that a monoid is both a $t$-B\'ezout monoid (i.e., a GCD-monoid) and a $t$-SP-monoid if and only if the radical of every principal ideal is principal.

After that we study the monoid of $r$-invertible $r$-ideals in Section 5. In particular, we characterize when every principal ideal of the monoid of $r$-invertible $r$-ideals is a finite product of pairwise comparable radical principal ideals. We also give a technical characterization of radical factorial monoids (i.e., monoids for which every principal ideal is a finite product of radical principal ideals). Furthermore, we describe when the monoid of $r$-invertible $r$-ideals is radical factorial.

Finally, we investigate several ring-theoretical constructions in Section 6. We show that if $R$ is an integral domain and $H$ is a grading monoid (i.e., a cancellative torsionless monoid), then $R[H]$ is a $w$-SP-domain if and only if $R$ is a $w$-SP-domain, $H$ is a $w$-SP-monoid and the homogeneous field of quotients of $R[H]$ is radical factorial. We also show that if $\ast$ is a star operation of finite type of an integral domain $R$, then the $\ast$-Nagata ring of $R$ is an SP-domain if and only if $R$ is a $\ast$-almost Dedekind $\ast$-SP-domain.

\section{Ideal systems}

In this section we introduce the notion of (finitary) ideal systems and the most important terminology. In the following a monoid $H$ is always a commutative semigroup with identity and more than one element such that every nonzero element of $H$ is cancellative. If not stated otherwise, then $H$ is written multiplicatively.

\begin{center}
\textit{Throughout this paper let $H$ be a monoid and let $G$ be the quotient monoid of $H$.}
\end{center}

Let $z(H)$ denote the set of {\it zero elements} of $H$ (i.e., $z(H)=\{z\in H\mid zx=z$ for all $x\in H\}$). (We introduce this notion to handle both monoids with and without a zero element. Also note that $|z(H)|\leq 1$.)

Let $X\subseteq H$ and $Y\subseteq G$. Set $\sqrt{X}=\{x\in H\mid x^n\in X$ for some $n\in\mathbb{N}\}$, called the {\it radical} of $X$ and $Y^{-1}=\{z\in G\mid zY\subseteq H\}$. We say that $X$ is an {\it $s$-ideal} of $H$ if $X=XH\cup z(H)$ and we say that $X$ is {\it radical} if $\sqrt{X}=X$. An $s$-ideal $J$ of $H$ is called a {\it principal ideal} of $H$ if it is generated by at most one element (i.e., $J=AH\cup z(H)$ for some $A\subseteq H$ with $|A|\leq 1$).

If $a\in H$, then $a$ is called {\it prime $($primary, radical$)$} if $aH$ (i.e., the principal ideal generated by $a$) is a prime (primary, radical) $s$-ideal of $H$. Let $\mathfrak X(H)$ denote the set of minimal prime $s$-ideals of $H$ which properly contain $z(H)$ and let $\mathcal{P}(X)$ denote the set of prime $s$-ideals of $H$ that are minimal above $X$.

\smallskip
By $H^{\bullet}$ (resp. $H^{\times}$) we denote the set of nonzero elements of $H$ (resp. the set of units of $H$) and by $\mathbb{P}(H)$ we denote the power set of $H$. Let $r:\mathbb{P}(H)\rightarrow\mathbb{P}(H)$, $X\mapsto X_r$ be a map. For subsets $X,Y\subseteq H$ and $c\in H$ we consider the following properties:
\begin{itemize}
\item[(A)] $XH\cup z(H)\subseteq X_r$.
\item[(B)] If $X\subseteq Y_r$, then $X_r\subseteq Y_r$.
\item[(C)] $cX_r=(cX)_r$.
\item[(D)] $X_r=\bigcup_{E\subseteq X,|E|<\infty} E_r$.
\end{itemize}

We say that $r$ is a {\it $($finitary$)$ ideal system} on $H$ if $r$ satisfies properties $A$, $B$, $C$ (and $D$) for all $X,Y\subseteq H$ and $c\in H$. Also note that an ideal system $r$ is finitary if and only if $X_r\subseteq\bigcup_{E\subseteq X,|E|<\infty} E_r$ for all $X\subseteq H$. Furthermore, if $r$ is an ideal system, then it follows from (A) and (B) that $r$ is idempotent (i.e., $(X_r)_r=X_r$ for each $X\subseteq H$).

\smallskip
Let $r$ be finitary ideal system on $H$ and $X\subseteq H$. We say that $X$ is an {\it $r$-ideal} (resp. an {\it $r$-invertible $r$-ideal}) if $X_r=X$ (resp. if $X_r=X$ and $(XX^{-1})_r=H$). Now let $I$ be an $r$-ideal of $H$. The $r$-ideal $I$ is called {\it nontrivial} if $z(H)\subsetneq I$ and it is called {\it proper} if $I\subsetneq H$. By $\mathcal{I}_r(H)$ (resp. $\mathcal{I}_r^*(H)$) we denote the set of $r$-ideals (resp. the set of $r$-invertible $r$-ideals) of $H$. Observe that $\sqrt{I}=\bigcap_{P\in\mathcal{P}(I)} P$ and $\mathcal{P}(I)\subseteq\mathcal{I}_r(H)$. If $I$ and $J$ are $r$-ideals of $H$, then $(IJ)_r$ is called the {\it $r$-product} of $I$ and $J$. Note that the set of $r$-ideals forms a commutative semigroup with identity under $r$-multiplication and the set of $r$-invertible $r$-ideals of $H$ forms a monoid under $r$-multiplication.

\smallskip
Note that every (nontrivial) principal ideal of $H$ is an ($r$-invertible) $r$-ideal of $H$. Let $\mathcal{H}$ be the set of nontrivial principal ideals of $H$, $q(\mathcal{I}_r^*(H))$, resp. $q(\mathcal{H})$, the quotient group of $\mathcal{I}_r^*(H)$, resp. $\mathcal{H}$, and $\mathcal{C}_r(H)=q(\mathcal{I}_r^*(H))/q(\mathcal{H})$, called the {\it $r$-class group} of $H$. Note that $\mathcal{C}_r(H)$ is trivial if and only if every $r$-invertible $r$-ideal of $H$ is principal. Moreover, $\mathcal{C}_r(H)$ is torsionfree if and only if for all $k\in\mathbb{N}$ and $I\in\mathcal{I}_r^*(H)$ such that $(I^k)_r$ is principal, it follows that $I$ is principal. Let $r$-${\rm spec}(H)$, resp. $r$-$\max(H)$ denote the set of prime $r$-ideals, resp. the set of $r$-maximal $r$-ideals of $H$. We say that $I\in\mathcal{I}_r(H)$ is $r$-finitely generated if $I=E_r$ for some finite $E\subseteq I$.

\smallskip
We say that $r$ is {\it modular} if for all $r$-ideals $I,J,N$ of $H$ with $I\subseteq N$ it follows that $(I\cup J)_r\cap N\subseteq (I\cup (J\cap N))_r$ (equivalently, for all $r$-ideals $I,J,N$ of $H$ with $I\subseteq N$ it follows that $(I\cup J)_r\cap N=(I\cup (J\cap N))_r$). Now let $p$ be a finitary ideal system on $H$. The ideal system $p$ is called finer than $r$ (or $r$ is called coarser than $p$), denoted by $p\leq r$, if $X_p\subseteq X_r$ for all $X\subseteq H$ (equivalently, every $r$-ideal of $H$ is a $p$-ideal of $H$). The notions of finer and coarser can be extended to arbitrary ideal systems.

\smallskip
Next we introduce the most important ideal systems. Let $T\subseteq H^{\bullet}$ be multiplicatively closed (i.e., $1\in T$ and $xy\in T$ for all $x,y\in T$). Then there is a unique finitary ideal system $T^{-1}r$ defined on $T^{-1}H$ such that $T^{-1}(X_r)=(T^{-1}X)_{T^{-1}r}$ for all $X\subseteq H$. Furthermore, if $r$ is modular, then $T^{-1}r$ is modular. If $P$ is a prime $s$-ideal of $H$, then we set $r_P=(H\setminus P)^{-1}r$. First we define the $s$-system.

\[
\textnormal{Let }s:\mathbb{P}(H)\rightarrow\mathbb{P}(H), X\mapsto XH\cup z(H)\textnormal{.}
\]

Note that $s$ is a finitary ideal system on $H$. Next we introduce the $v$-system and the $t$-system. First let $H=G$.
\[
\textnormal{For }X\subseteq H\textnormal{ let }X_v=X_t=z(H)\textnormal{ if }X\subseteq z(H)\textnormal{ and }X_v=X_t=H\textnormal{ if }X\nsubseteq z(H)\textnormal{.}
\]

Now let $H\not=G$.
\[
\textnormal{Let }v:\mathbb{P}(H)\rightarrow\mathbb{P}(H), X\mapsto (X^{-1})^{-1}\textnormal{ and }t:\mathbb{P}(H)\rightarrow\mathbb{P}(H), X\mapsto\bigcup_{E\subseteq X, |E|<\infty} E_v\textnormal{.}
\]

A subset $A$ of $H$ is called a divisorial ideal if $A_v=A$. Note that every $r$-invertible $r$-ideal of $H$ is divisorial. Let $R$ be an integral domain. Now we define the $d$-system.

\[
\textnormal{Let }d:\mathbb{P}(R)\rightarrow\mathbb{P}(R), X\mapsto {}_R(X)\textnormal{,}
\]
where ${}_R(X)$ is the ring ideal generated by $X$. Now let $p\leq r$. Next we introduce a finitary ideal system $\widetilde{r_p}$ depending on $p$ and $r$. We study some of its elementary properties in Lemma~\ref{Lem 2.1}.

\[
\textnormal{Let }\widetilde{r_p}:\mathbb{P}(H)\rightarrow\mathbb{P}(H), X\mapsto\{x\in H\mid xF\subseteq X_p\textnormal{ and }F_r=H\textnormal{ for some }F\subseteq H\}\textnormal{.}
\]

\begin{lemma}\label{Lem 2.1} Let $p$ and $r$ be finitary ideal systems on $H$ such that $p\leq r$.
\begin{itemize}
\item[(1)] $\widetilde{r_p}$ is a finitary ideal system on $H$ such that $p\leq\widetilde{r_p}\leq r$.
\item[(2)] $r$-$\max(H)=\widetilde{r_p}$-$\max(H)$ and $X_{\widetilde{r_p}}=\bigcap_{M\in r\textnormal{-}\max(H)} (X_p)_M$ for each $X\subseteq H$.
\item[(3)] $\mathcal{I}_r^*(H)=\mathcal{I}_{\widetilde{r_p}}^*(H)$, as monoids.
\item[(4)] If $p$ is modular, then $\widetilde{r_p}$ is modular.
\item[(5)] If $m$ and $n$ are finitary ideal systems on $H$ such that $p\leq m\leq\widetilde{r_p}\leq n\leq r$, then $\widetilde{n_m}=\widetilde{r_p}$.
\end{itemize}
\end{lemma}

\begin{proof} (1) \textsc{Claim 1.} If $Y\subseteq H$ and $N$ is a finite subset of $Y_{\widetilde{r_p}}$, then $NF\subseteq Y_p$ and $F_r=H$ for some $F\subseteq H$.

\smallskip
Let $Y\subseteq H$ and let $N$ be a finite subset of $Y_{\widetilde{r_p}}$. For each $e\in N$, there is some subset $F_e$ of $H$ such that $eF_e\subseteq Y_p$ and $(F_e)_r=H$. Set $F=\prod_{e\in N} F_e$. Then $NF\subseteq Y_p$ and $F_r=H$.\qed(Claim 1)

\smallskip
\textsc{Claim 2.} If $X\subseteq H$ and $x\in X_{\widetilde{r_p}}$, then there are some finite $E\subseteq H$ and some finite $N\subseteq X$ such that $xE\subseteq N_p$, $E_r=H$, $x\in N_{\widetilde{r_p}}$ and $x\in X_r$.

\smallskip
Let $X\subseteq H$ and $x\in X_{\widetilde{r_p}}$. There is some $E\subseteq H$ such that $xE\subseteq X_p$ and $E_r=H$. Since $r$ is finitary, we can assume without restriction that $E$ is finite. Since $p$ is finitary, there is some finite $N\subseteq X$ such that $xE\subseteq N_p$. Consequently, $x\in N_{\widetilde{r_p}}$ and $x\in xH=xE_r=(xE)_r\subseteq (X_p)_r=X_r$.\qed(Claim 2)

\smallskip
Let $X,Y\subseteq H$ and $c\in H$. If $y\in X_p$, then since $\{1\}_r=H$ and $y\in y\{1\}\subseteq X_p$, we have that $y\in X_{\widetilde{r_p}}$. Therefore, $X_p\subseteq X_{\widetilde{r_p}}$, and hence $X_s\subseteq X_{\widetilde{r_p}}$.

\smallskip
Next we show that if $X\subseteq Y_{\widetilde{r_p}}$, then $X_{\widetilde{r_p}}\subseteq Y_{\widetilde{r_p}}$. Let $X\subseteq Y_{\widetilde{r_p}}$ and $x\in X_{\widetilde{r_p}}$. By Claim 2 there are some finite $E\subseteq H$ and some finite $N\subseteq X$ such that $xE\subseteq N_p$ and $E_r=H$. By Claim 1 there is some $F\subseteq H$ such that $NF\subseteq Y_p$ and $F_r=H$. This implies that $xEF\subseteq N_pF\subseteq (N_pF)_p=(NF)_p\subseteq Y_p$ and $(EF)_r=H$, and hence $x\in Y_{\widetilde{r_p}}$.

\smallskip
Now we show that $cX_{\widetilde{r_p}}=(cX)_{\widetilde{r_p}}$. First let $z\in X_{\widetilde{r_p}}$. There is some $E\subseteq H$ such that $zE\subseteq X_p$ and $E_r=H$. Since $czE\subseteq cX_p=(cX)_p$ and $E_r=H$, we have that $cz\in (cX)_{\widetilde{r_p}}$. Therefore, $cX_{\widetilde{r_p}}\subseteq (cX)_{\widetilde{r_p}}$.

Now let $z\in (cX)_{\widetilde{r_p}}$. It follows by Claim 2 that $z\in (cX)_r=cX_r$, and hence $z=cv$ for some $v\in H$. If $c\in z(H)$, then $z\in z(H)\subseteq cX_{\widetilde{r_p}}$. Now let $c\not\in z(H)$. There is some $E\subseteq H$ such that $cvE\subseteq (cX)_p=cX_p$ and $E_r=H$. Consequently, $vE\subseteq X_p$, and thus $v\in X_{\widetilde{r_p}}$. We infer that $z\in cX_{\widetilde{r_p}}$.

\smallskip
Putting all these parts together shows that $\widetilde{r_p}$ is an ideal system on $H$. We infer by Claim 2 that $X_{\widetilde{r_p}}\subseteq\bigcup_{F\subseteq X,|F|<\infty} F_{\widetilde{r_p}}$, and hence $\widetilde{r_p}$ is finitary. We have already shown (below the proof of Claim 2) that $X_p\subseteq X_{\widetilde{r_p}}$. Moreover, $X_{\widetilde{r_p}}\subseteq X_r$ by Claim 2. This implies that $p\leq\widetilde{r_p}\leq r$.

\smallskip
(2) To show that $r$-$\max(H)=\widetilde{r_p}$-$\max(H)$ it is sufficient to show that every $M\in\widetilde{r_p}$-$\max(H)$ is an $r$-ideal of $H$. Let $M\in\widetilde{r_p}$-$\max(H)$. Assume that $M$ is not an $r$-ideal of $H$. Then $M_r=H$. We have that $1M\subseteq M_p$, and hence $1\in M_{\widetilde{r_p}}=M$, a contradiction.

Now let $X\subseteq H$. Let $x\in X_{\widetilde{r_p}}$ and $N\in r$-$\max(H)$. Then $xE\subseteq X_p$ and $E_r=H$ for some $E\subseteq H$, and thus there is some $y\in E\setminus N$. It follows that $xy\in X_p$, and hence $x\in y^{-1}X_p\subseteq (X_p)_N$. This implies that $X_{\widetilde{r_p}}\subseteq (X_p)_N$ for every $N\in r$-$\max(H)$. Moreover, we have that
\[
\bigcap_{M\in r\textnormal{-}\max(H)} (X_p)_M=\bigcap_{M\in\widetilde{r_p}\textnormal{-}\max(H)} (X_p)_M\subseteq\bigcap_{M\in\widetilde{r_p}\textnormal{-}\max(H)} (X_{\widetilde{r_p}})_M=X_{\widetilde{r_p}}.
\]

(3) Since $\widetilde{r_p}\leq r$ by (1), we have clearly that $\mathcal{I}_{\widetilde{r_p}}^*(H)\subseteq\mathcal{I}_r^*(H)$. Now let $I\in\mathcal{I}_r^*(H)$. Assume that $I\not\in\mathcal{I}_{\widetilde{r_p}}^*(H)$. Then $(II^{-1})_{\widetilde{r_p}}\subsetneq H$, and hence there is some $M\in\widetilde{r_p}$-$\max(H)$ such that $II^{-1}\subseteq M$. We infer by (2) that $M\in r$-$\max(H)$, and hence $H=(II^{-1})_r\subseteq M_r=M$, a contradiction. It remains to show that the $\widetilde{r_p}$-multiplication and the $r$-multiplication coincide on $\mathcal{I}_{\widetilde{r_p}}^*(H)$. Let $J,L\in\mathcal{I}_{\widetilde{r_p}}^*(H)$. Then $(JL)_{\widetilde{r_p}}\in\mathcal{I}_{\widetilde{r_p}}^*(H)\subseteq\mathcal{I}_r(H)$. We infer that $(JL)_{\widetilde{r_p}}=((JL)_{\widetilde{r_p}})_r=(JL)_r$, since $\widetilde{r_p}\leq r$ by (1).

(4) Let $p$ be modular and let $I,J,N$ be $\widetilde{r_p}$-ideals of $H$ such that $I\subseteq N$. Let $x\in (I\cup J)_{\widetilde{r_p}}\cap N$. Then there is some $E\subseteq H$ such that $xE\subseteq (I\cup J)_p$ and $E_r=H$, and thus $xE\subseteq (I\cup J)_p\cap N=(I\cup (J\cap N))_p$. We infer that $x\in (I\cup (J\cap N))_{\widetilde{r_p}}$.

(5) Let $m$ and $n$ be finitary ideal systems on $H$ such that $p\leq m\leq\widetilde{r_p}\leq n\leq r$ and $X\subseteq H$. First let $x\in X_{\widetilde{n_m}}$. Then there is some finite $E\subseteq H$ such that $xE\subseteq X_m$ and $E_n=H$. Then $xE\subseteq X_{\widetilde{r_p}}$ and $E_r=H$. As shown in (1), we have that $xEF\subseteq X_p$ for some $F\subseteq H$ with $F_r=H$. Observe that $(EF)_r=H$, and thus $x\in X_{\widetilde{r_p}}$.

Now let $x\in X_{\widetilde{r_p}}$. Then $xE\subseteq X_p$ and $E_r=H$ for some $E\subseteq H$. By (2) we have that $E_{\widetilde{r_p}}=H$. Therefore, $xE\subseteq X_m$ and $E_n=H$, and hence $x\in X_{\widetilde{n_m}}$.
\end{proof}

If $p\leq r$ are finitary ideal systems on $H$ and $p$ is modular, then we say (in view of Lemma~\ref{Lem 2.1}(4)) that $\widetilde{r_p}$ is the {\it $p$-modularization} of $r$. Set $w_p=\widetilde{t_p}$ and $w=w_s$. We have that $s$ is the finest ideal system on $H$, $t$ is the coarsest finitary ideal system on $H$ and $v$ is the coarsest ideal system on $H$. Furthermore, $s\leq w\leq t\leq v$ and if $H$ is an integral domain, then $s\leq d\leq w_d\leq t\leq v$. Note that both the $s$-system and the $d$-system are modular finitary ideal systems. In what follows, we use the remarks of this paragraph without further citation.

\section{Results for finitary ideal systems}

Let $r$ be a finitary ideal system on $H$. We say that $H$ is an {\it $r$-SP-monoid} if every $r$-ideal of $H$ is a finite $r$-product of radical $r$-ideals of $H$. Moreover, $H$ is called {\it radical factorial} if every principal ideal of $H$ is a finite product of radical principal ideals of $H$. Furthermore, $H$ is called {\it factorial} if every principal ideal of $H$ is a finite product of prime principal ideals of $H$ (equivalently, every nontrivial prime $t$-ideal of $H$ contains a nontrivial prime principal ideal of $H$). We say that $H$ is a {\it valuation monoid} if the principal ideals of $H$ are pairwise comparable (equivalently, the $s$-ideals of $H$ are pairwise comparable). Also note that if $H$ is a valuation monoid, then $s=r=t$ (i.e., the $s$-system is the unique finitary ideal system on $H$). Moreover, if $H\not=G$, then $H$ is called a {\it discrete valuation monoid} (or a {\it DVM}) if every $s$-ideal of $H$ is principal (equivalently, every prime $s$-ideal of $H$ is principal). We say that $H$ satisfies the {\it Principal Ideal Theorem} if for each nontrivial principal ideal $I$ of $H$ we have that $\mathcal{P}(I)\subseteq\mathfrak{X}(H)$. Finally, $H$ is called {\it $r$-local} if $H\setminus H^{\times}$ is an $r$-ideal of $H$ (equivalently, $|r$-$\max(H)|=1)$. Observe that if $H$ is $r$-local, then $\mathcal{C}_r(H)$ is trivial.

It is easy to see that if the radical of every nontrivial principal ideal of $H$ is $r$-invertible or every nontrivial principal ideal of $H$ is a finite $r$-product of radical $r$-ideals of $H$ (in particular if $H$ is radical factorial or an $r$-SP-monoid), then $H_M$ is radical factorial for each $M\in r$-$\max(H)$. (In the first case we can show that the radical of every principal ideal of $H_M$ is principal for each $M\in r$-$\max(H)$ and then apply \cite[Proposition 2.10]{R}.) The main purpose of this section is to present new characterizations of $r$-almost Dedekind monoids, $r$-almost Dedekind $r$-SP-monoids and $r$-B\'ezout $r$-SP-monoids.

\begin{proposition}\label{Prop 3.1} Let $r$ be a finitary ideal system on $H$ such that $H_M$ is radical factorial for each $M\in r$-$\max(H)$. Then $\bigcap_{P\in\mathfrak X(H)} H_P=H$, $H_Q$ is a DVM for each $Q\in\mathfrak X(H)$ and $\mathcal{P}(I)\subseteq\mathfrak X(H)$ for each $r$-invertible $r$-ideal $I$ of $H$.
\end{proposition}

\begin{proof} By \cite[Proposition 2.4]{R} we have for each $M\in r$-$\max(H)$ that $H_M=\bigcap_{P\in\mathfrak X(H_M)} (H_M)_P$, $(H_M)_Q$ is a DVM for each $Q\in\mathfrak X(H_M)$ and $\mathcal{P}(xH_M)\subseteq\mathfrak X(H_M)$ for each $x\in H_M^{\bullet}$. It is easy to see that $\mathfrak X(H_M)=\{P_M\mid P\in\mathfrak X(H), P\subseteq M\}$ for each $M\in r$-$\max(H)$.

\smallskip
We prove that $\bigcap_{P\in\mathfrak X(H)} H_P=H$. If $M\in r$-$\max(H)$, then

\[
H_M=\bigcap_{Q\in\mathfrak X(H_M)} (H_M)_Q=\bigcap_{P\in\mathfrak X(H), P\subseteq M} (H_M)_{P_M}=\bigcap_{P\in\mathfrak X(H), P\subseteq M} H_P.
\]

It follows that

\[
H=\bigcap_{M\in r\textnormal{-}\max(H)} H_M=\bigcap_{M\in r\textnormal{-}\max(H)}\bigcap_{P\in\mathfrak X(H), P\subseteq M} H_P=\bigcap_{P\in\mathfrak X(H)} H_P.
\]

\smallskip
Let $Q\in\mathfrak X(H)$. Then $Q_M\in\mathfrak X(H_M)$, and hence $H_Q=(H_M)_{Q_M}$ is a DVM.

\smallskip
Finally we show that $\mathcal{P}(I)\subseteq\mathfrak X(H)$ for each $r$-invertible $r$-ideal $I$ of $H$. Let $I$ be an $r$-invertible $r$-ideal of $H$ and $P\in\mathcal{P}(I)$. There is some $M\in r$-$\max(H)$ such that $P\subseteq M$. Observe that $I_M$ is a nontrivial principal ideal of $H_M$ and $P_M\in\mathcal{P}(I_M)\subseteq\mathfrak X(H_M)$. Therefore, there is some $P^{\prime}\in\mathfrak X(H)$ such that $P^{\prime}\subseteq M$ and $P_M=(P^{\prime})_M$. This implies that $P=P_M\cap H=(P^{\prime})_M\cap H=P^{\prime}\in\mathfrak X(H)$.
\end{proof}

Let $r$ be a finitary ideal system on $H$. The monoid $H$ is called {\it $r$-treed} if for all $M\in r$-$\max(H)$, it follows that the prime $r$-ideals of $H$ that are contained in $M$ form a chain. Moreover, $H$ is called an {\it $r$-almost Dedekind monoid} (or an {\it almost $r$-Dedekind monoid} in the terminology of \cite{R}) if $H=G$ or if $H_M$ is a DVM for each $M\in r$-$\max(H)$.

\begin{lemma}\label{Lem 3.2} Let $r$ be a finitary ideal system on $H$ such that every nontrivial prime $r$-ideal of $H$ contains an $r$-invertible radical $r$-ideal of $H$.
\begin{itemize}
\item[(1)] If the prime $r$-ideals of $H$ form a chain and $H\not=G$, then $H$ is a DVM.
\item[(2)] If $H$ is $r$-treed, then $H$ is an $r$-almost Dedekind $r$-SP-monoid.
\end{itemize}
\end{lemma}

\begin{proof} (1) Let the prime $r$-ideals of $H$ form a chain and let $H\not=G$. Then $H$ is $r$-local, and thus every $r$-invertible $r$-ideal of $H$ is principal. Moreover, every radical $r$-ideal of $H$ is a prime $r$-ideal of $H$. Therefore, every nontrivial prime $r$-ideal of $H$ contains a nontrivial prime principal ideal of $H$. Let $\Omega$ be the set of all elements of $H$ which can be represented as a product of a unit of $H$ times a (possibly empty) finite product of nonzero prime elements of $H$. Assume that $H$ is not factorial. Then there is some nonzero $x\in H\setminus\Omega$. It is straightforward to show that $xH\cap\Omega=\emptyset$. Since $\Omega$ is a multiplicatively closed subset of $H$, $xH$ is an $r$-ideal of $H$ and $r$ is finitary, we infer that $xH\subseteq P$ and $P\cap\Omega=\emptyset$ for some prime $r$-ideal $P$ of $H$. Since $P$ contains a nonzero prime principal ideal of $H$, we have that $P\cap\Omega\not=\emptyset$, a contradiction. This implies that $H$ is a factorial monoid. Since the prime $r$-ideals of $H$ form a chain, we have that $|\mathfrak X(H)|=1$, and thus $H$ is a DVM.

(2) Let $H$ be $r$-treed and $M\in r$-$\max(H)$. Clearly, the prime $r_M$-ideals of $H_M$ form a chain and every nontrivial prime $r_M$-ideal of $H_M$ contains an $r_M$-invertible radical $r_M$-ideal of $H_M$. Therefore, $H_M$ is a DVM by (1), and hence $H$ is an $r$-almost Dedekind monoid. It follows from \cite[Corollary 3.4]{R} that $H$ is an $r$-SP-monoid.
\end{proof}

\begin{lemma}\label{Lem 3.3} Let $r$ be a finitary ideal system on $H$, $k\in\mathbb{N}$, $P\in\mathfrak X(H)$ and $I_i$ a nontrivial radical $r$-ideal of $H$ for each $i\in [1,k+1]$ such that $\bigcup_{i=1}^{k+1} I_i\subseteq P$. Then $(\prod_{i=1}^{k+1} I_i)_r$ does not contain the $k$-th power of any nonzero radical element of $H$.
\end{lemma}

\begin{proof} Suppose to the contrary that there is some nonzero radical element $x\in H$ such that $x^k\in (\prod_{i=1}^{k+1} I_i)_r$. We infer that $x\in P$. It follows that $P_P=xH_P=(I_j)_P$ for each $j\in [1,k+1]$, and hence $P_P^k=x^kH_P\subseteq ((\prod_{j=1}^{k+1} I_j)_r)_P=(\prod_{j=1}^{k+1} (I_j)_P)_{r_P}=(x^{k+1}H_P)_{r_P}=x^{k+1}H_P$. Therefore, $x^kH_P=x^{k+1}H_P$, and hence $x\in H_P^{\times}$, a contradiction.
\end{proof}

\begin{proposition}\label{Prop 3.4} Let $r$ be a finitary ideal system on $H$ and let the radical of every principal ideal of $H$ be principal.
\begin{itemize}
\item[(1)] For each nontrivial $r$-finitely generated $r$-ideal $I$ of $H$ there is some nonzero $z\in H$ such that $\{P\in\mathfrak X(H)\mid I\subseteq P\}=\{P\in\mathfrak X(H)\mid z\in P\}$.
\item[(2)] $\mathcal{C}_r(H)$ is trivial.
\end{itemize}
\end{proposition}

\begin{proof} (1) \textsc{Claim 1.} If $a,b$ are nonzero radical elements of $H$ such that $b$ divides $a$, then

\[
\left\{P\in\mathfrak X(H)\mid\frac{a}{b}\in P\right\}=\{P\in\mathfrak X(H)\mid a\in P, b\not\in P\}.
\]

\smallskip
To prove Claim 1 let $a,b\in H$ be nonzero radical elements of $H$ such that $b$ divides $a$. First let $P\in\mathfrak X(H)$ be such that $\frac{a}{b}\in P$. It is obvious that $a\in P$. Since $aH_P$ is a nonzero radical ideal of $H_P$ we have that $aH_P=P_P$. Suppose that $b\in P$. Then $a\in P^2$, and hence $aH_P=P_P=P_P^2=a^2H_P$. Therefore, $a\in H_P^{\times}$, a contradiction. We infer that $b\not\in P$. The converse inclusion is trivially satisfied.\qed(Claim 1)

\smallskip
\textsc{Claim 2.} For all nonzero $x,y\in H$ there is some nonzero $z\in H$ such that $\{P\in\mathfrak X(H)\mid (xH\cup yH)_r\subseteq P\}=\{P\in\mathfrak X(H)\mid z\in P\}$.

\smallskip
To prove Claim 2 let $x,y\in H$ be nonzero. There exist nonzero radical elements $a,b,c\in H$ such that $\sqrt{xyH}=aH$, $\sqrt{xH}=bH$ and $\sqrt{yH}=cH$. We have that $aH=bH\cap cH$. Moreover, $\sqrt{\frac{a}{b}H}\cap\sqrt{\frac{a}{c}H}=\sqrt{\frac{a^2}{bc}H}=dH$ for some nonzero radical element $d\in H$. Set $z=\frac{a}{d}$.

It follows by Claim 1 that $\{P\in\mathfrak X(H)\mid d\in P\}=\{P\in\mathfrak X(H)\mid\frac{a}{b}\in P\}\cup\{P\in\mathfrak X(H)\mid\frac{a}{c}\in P\}=\{P\in\mathfrak X(H)\mid a\in P, (b\not\in P$ or $c\not\in P)\}$.

We infer by Claim 1 that $\{P\in\mathfrak X(H)\mid x,y\in P\}=\{P\in\mathfrak X(H)\mid b,c\in P\}=\{P\in\mathfrak X(H)\mid a\in P, d\not\in P\}=\{P\in\mathfrak X(H)\mid z\in P\}$.\qed(Claim 2)

The statement now follows by induction from Claim 2.

\smallskip
(2) \textsc{Claim.} The radical of every $r$-invertible $r$-ideal of $H$ is principal.

\smallskip
To prove the claim let $I$ be an $r$-invertible $r$-ideal of $H$. By Proposition~\ref{Prop 3.1}, we have that $\mathcal{P}(I)\subseteq\mathfrak X(H)$. It follows by (1) that there is some nonzero $z\in H$ such that $\mathcal{P}(I)=\{P\in\mathfrak X(H)\mid I\subseteq P\}=\{P\in\mathfrak X(H)\mid z\in P\}=\mathcal{P}(zH)$, and hence $\sqrt{I}=\sqrt{zH}$ is a principal ideal of $H$.\qed(Claim)

\smallskip
Now let $J$ be an $r$-invertible $r$-ideal of $H$. By the claim there is some nonzero radical $z_1\in H$ such that $\sqrt{J}=z_1H$. Therefore, $z_1^k\in J$ for some $k\in\mathbb{N}$.

Next we recursively construct nonzero radical elements $z_i$ of $H$ such that $z_iH=\sqrt{(\prod_{j=1}^{i-1} z_j)^{-1}J}$ for each $i\in [1,k+1]$. Note that $z_1H=\sqrt{(\prod_{j=1}^{1-1} z_j)^{-1}J}$. Now let $i\in [1,k]$ and suppose that we have already constructed the first $i$ elements. It follows that $(\prod_{j=1}^{i-1} z_j)^{-1}J\subseteq z_iH$, and thus $(\prod_{j=1}^i z_j)^{-1}J\subseteq H$. Set $L=(\prod_{j=1}^i z_j)^{-1}J$. Then $(\prod_{j=1}^i z_j)L_r=((\prod_{j=1}^i z_j)L)_r=J_r=J$, and hence $L$ is an $r$-ideal of $H$. Since  $J=(L\prod_{j=1}^i z_jH)_r$ and $J$ is $r$-invertible, we infer that $L$ is $r$-invertible. By the claim there is some nonzero radical $z_{i+1}\in H$ such that $\sqrt{L}=z_{i+1}H$. This completes the construction.

Assume that $z_{k+1}\not\in H^{\times}$. Then there is some $P\in\mathcal{P}(z_{k+1}H)$. It follows from Proposition~\ref{Prop 3.1} that $P\in\mathfrak{X}(H)$. Observe that $z_iH\subseteq z_{i+1}H$ for each $i\in [1,k]$, and thus $\bigcup_{i=1}^{k+1} z_iH\subseteq P$. Moreover, we have that $z_1^k\in J\subseteq\prod_{j=1}^{k+1} z_jH=(\prod_{j=1}^{k+1} z_jH)_r$, which contradicts Lemma~\ref{Lem 3.3}. Therefore, $z_{k+1}\in H^{\times}$, and hence $\sqrt{(\prod_{j=1}^k z_j)^{-1}J}=z_{k+1}H=H$. This implies that $(\prod_{j=1}^k z_j)^{-1}J=H$. Consequently, $J=(\prod_{j=1}^k z_j)H$ is a principal ideal of $H$.
\end{proof}

\begin{proposition}\label{Prop 3.5} Let $H\not=G$ and $r$ a finitary ideal system on $H$ and let $H$ be $r$-local such that the radical of every $r$-finitely generated $r$-ideal of $H$ is principal. Then $H$ is a DVM.
\end{proposition}

\begin{proof} By Proposition~\ref{Prop 3.1}, $H$ satisfies the Principal Ideal Theorem. Assume that $H$ is not a valuation monoid. Then there exist $x,y\in H$ such that $xH\nsubseteq yH$ and $yH\nsubseteq xH$. Using the fact that the radical of every $r$-finitely generated $r$-ideal of $H$ is principal, we can recursively construct nonzero radical elements $z_i$ of $H$ such that for every $i\in\mathbb{N}$,

\[
\prod_{j=1}^{i-1} z_j\mid x\textnormal{, }\prod_{j=1}^{i-1} z_j\mid y\textnormal{ and }z_iH=\sqrt{\Bigg(\frac{x}{\prod_{j=1}^{i-1} z_j}H\cup\frac{y}{\prod_{j=1}^{i-1} z_j}H\Bigg)_r}.
\]

For $i\in\mathbb{N}$ set $w_i=\frac{x}{\prod_{j=1}^{i-1} z_j}$ and $v_i=\frac{y}{\prod_{j=1}^{i-1} z_j}$. Observe that if $i\in\mathbb{N}$, then $w_iH$ and $v_iH$ are not comparable, and hence $w_i,v_i\in H\setminus H^{\times}$. Since $H$ is $r$-local, we infer that $z_i\in H\setminus H^{\times}$ for all $i\in\mathbb{N}$. There is some $k\in\mathbb{N}$ such that $z_1^k\in (xH\cup yH)_r=(\prod_{i=1}^k z_i)(w_{k+1}H\cup v_{k+1}H)_r\subseteq (\prod_{i=1}^{k+1} z_iH)_r$. Also note that $\bigcup_{i=1}^{k+1} z_iH=z_{k+1}H\subseteq P$ for some $P\in\mathfrak X(H)$, which contradicts Lemma~\ref{Lem 3.3}.

Consequently, $H$ is a valuation monoid. It follows by Lemma~\ref{Lem 3.2}(1) that $H$ is a DVM.
\end{proof}

Let $r$ be a finitary ideal system on $H$. We say that $H$ satisfies the {\it $r$-prime power condition} if every primary $r$-ideal of $H$ is an $r$-power of its radical. Note that every $r$-SP-monoid satisfies the $r$-prime power condition (see \cite[Proposition 3.10(1)]{R}). Moreover, $H$ satisfies the {\it strong $r$-prime power condition} if every $r$-ideal of $H$ with prime radical is an $r$-power of its radical. Finally, $H$ is called {\it primary $r$-ideal inclusive} if for all $P,Q\in r$-${\rm spec}(H)$ such that $P\subsetneq Q$ it follows that $P\subseteq I\subsetneq\sqrt{I}\subseteq Q$ for some primary $r$-ideal $I$ of $H$. Now let $I$ be an $r$-ideal of $H$. We say that $I$ is {\it $r$-cancellative} if for all $r$-ideals $J$ and $L$ of $H$ such that $(IJ)_r=(IL)_r$ it follows that $J=L$. Moreover, $I$ is called {\it $r$-half cancellative} (or {\it $r$-unit-cancellative}) if for all $J\in\mathcal{I}_r(H)$ with $I=(IJ)_r$ it follows that $J=H$.

\smallskip
Let $T\subseteq H^{\bullet}$ a multiplicatively closed subset. Note that if $H$ satisfies the (strong) $r$-prime power condition, then $T^{-1}H$ satisfies the (strong) $T^{-1}r$-prime power condition. Moreover, if $H$ is primary $r$-ideal inclusive, then $T^{-1}H$ is primary $T^{-1}r$-ideal inclusive. (By \cite[Lemma 3.8]{R} it remains to show that if $H$ satisfies the strong $r$-prime power condition, then $T^{-1}H$ satisfies the strong $T^{-1}r$-prime power condition. Let $H$ satisfy the strong $r$-prime power condition and let $J$ be a $T^{-1}r$-ideal of $T^{-1}H$ with prime radical. Set $I=J\cap H$. Then $I$ is an $r$-ideal of $H$ and $J=T^{-1}I$. Since $\sqrt[T^{-1}H]{J}$ is a prime $T^{-1}r$-ideal of $T^{-1}H$, we have that $\sqrt{I}=\sqrt[T^{-1}H]{J}\cap H$ is a prime $r$-ideal of $H$. Observe that $T^{-1}\sqrt{I}=\sqrt[T^{-1}H]{T^{-1}I}=\sqrt[T^{-1}H]{J}$. Therefore, $I=((\sqrt{I})^k)_r$ for some $k\in\mathbb{N}$, and thus $J=T^{-1}I=((\sqrt[T^{-1}H]{J})^k)_{T^{-1}r}$.) In what follows, we use the remarks of this paragraph without further citation.

\begin{proposition}\label{Prop 3.6} $[$cf. \cite[Theorem 2.14 and Proposition 2.16]{BFP}, \cite[Theorem 1.1]{BP} and \cite[Theorems 4.5 and 4.6]{K}$]$ Let $H\not=G$ and $r$ a finitary ideal system on $H$. The following are equivalent:
\begin{itemize}
\item[(1)] $H$ is an $r$-almost Dedekind monoid.
\item[(2)] $H$ satisfies the strong $r$-prime power condition and every nontrivial $r$-ideal of $H$ is $r$-cancellative.
\item[(3)] For all nonzero $x\in H$ and $P\in\mathcal{P}(xH)$, $P$ is $r$-half cancellative and every $r$-ideal of $H$ whose radical is $P$ is an $r$-power of its radical.
\item[(4)] $H$ is $r$-treed and satisfies the strong $r$-prime power condition.
\item[(5)] $H$ satisfies the strong $r$-prime power condition and $r$ is modular.
\item[(6)] $r$-$\max(H)=\mathfrak X(H)$ and $H$ satisfies the $r$-prime power condition.
\item[(7)] $H$ satisfies the $r$-prime power condition and the Principal Ideal Theorem, and $H$ is primary $r$-ideal inclusive.
\end{itemize}
\end{proposition}

\begin{proof} \textsc{Claim 1.} If $H$ satisfies the $r$-prime power condition and $P\in\mathcal{P}(xH)$ for some nonzero $x\in H$, then $P_P$ is principal.

Let $x\in H$ be nonzero and $P\in\mathcal{P}(xH)$. Since $P_P$ is the only prime $s$-ideal of $H_P$ such that $x\in P_P$, we infer that $\sqrt[\uproot{3}H_P]{xH_P}=P_P$. Note that $H_P$ satisfies the $r_P$-prime power condition (by the discussion above), and thus $xH_P=(P_P^k)_{r_P}$ for some $k\in\mathbb{N}$. (Note that $P_P\in r_P$-$\max(H_P)$, and thus $xH_P$ is $P_P$-primary.) Therefore, $P_P$ is $r_P$-invertible, and hence $P_P$ is principal, since $H_P$ is $r_P$-local.\qed(Claim 1)

\smallskip
(1) $\Rightarrow$ (2),(5): Clearly, $r$-$\max(H)=\mathfrak X(H)$. Let $I$ be a nontrivial $r$-ideal of $H$ and $J,L$ $r$-ideals of $H$ such that $(IJ)_r=(IL)_r$. If $M\in r$-$\max(H)$, then $I_M=xH_M$ for some nonzero $x\in H_M$ and hence $xJ_M=I_MJ_M=(I_MJ_M)_{r_M}=((IJ)_r)_M=((IL)_r)_M=(I_ML_M)_{r_M}=I_ML_M=xJ_M$. We infer that $J_M=L_M$ for each $M\in r$-$\max(H)$, and thus $J=L$.

Now let $I$ be a nontrivial $r$-ideal of $H$ with prime radical. Set $M=\sqrt{I}$. Observe that $H_M$ is a DVM, and thus every nontrivial $s$-ideal of $H_M$ is a power of $M_M$. Consequently, $I_M=M_M^k=(M_M^k)_{r_M}=((M^k)_r)_M$ for some $k\in\mathbb{N}$. Since $M\in r$-$\max(H)$, both $I$ and $(M^k)_r$ are $M$-primary $r$-ideals of $H$, and hence $I=I_M\cap H=((M^k)_r)_M\cap H=(M^k)_r$. It is clear that $r$ is modular.

(2) $\Rightarrow$ (3): This is obvious.

(3) $\Rightarrow$ (4): It is sufficient to show that every $r$-maximal $r$-ideal of $H$ is of height one. Let $M$ be an $r$-maximal $r$-ideal of $H$. Assume that $M$ is not of height one, then there exist $x\in M\setminus\{0\}$ and $P\in\mathcal{P}(xH)$ such that $P\subsetneq M$. By Claim 1 there is some $y\in P_P$ such that $P_P=yH_P$. We have that $\sqrt{(PM)_r}=\sqrt{P}\cap\sqrt{M}=P$, and thus $(PM)_r=(P^k)_r$ for some $k\in\mathbb{N}$. Since $(P^2)_r\subseteq (PM)_r\subseteq P$ and $P$ is $r$-half cancellative, we infer that $(PM)_r=(P^2)_r$, and thus $yH_P=(P_PM_P)_{r_P}=(P_P^2)_{r_P}=y^2H_P$. This implies that $P_P=yH_P=H_P$, a contradiction.

(4) $\Rightarrow$ (6): It is sufficient to show that every $r$-maximal $r$-ideal of $H$ is of height one. Let $M$ be an $r$-maximal $r$-ideal of $H$. First we show that the radical of every nontrivial principal ideal of $H_M$ is $r_M$-invertible. Let $I$ be a nontrivial proper principal ideal of $H_M$. The prime $r_M$-ideals of $H_M$ form a chain, and hence $\sqrt[\uproot{3}H_M]{I}$ is a prime $r_M$-ideal of $H_M$. Since $H_M$ satisfies the strong $r_M$-prime power condition, we have that $I=((\sqrt[\uproot{3}H_M]{I})^k)_{r_M}$ for some $k\in\mathbb{N}$. This implies that $\sqrt[\uproot{3}H_M]{I}$ is $r_M$-invertible.

We infer that every nontrivial prime $r_M$-ideal of $H_M$ contains an $r_M$-invertible radical $r_M$-ideal of $H_M$. It follows by Lemma~\ref{Lem 3.2}(1) that $H_M$ is a DVM, and hence $M\in\mathfrak{X}(H)$.

(5) $\Rightarrow$ (6): Assume that $r$-$\max(H)\not=\mathfrak X(H)$. Then there exist $y\in H^{\bullet}$, $P\in\mathcal{P}(yH)$ and $M\in r$-$\max(H)$ such that $P\subsetneq M$. By Claim 1 there is some $x\in P$ such that $P_P=xH_P$. Observe that $\sqrt{((P^2)_r\cup xM)_r}=P$, and hence $((P^2)_r\cup xM)_r=(P^k)_r$ for some $k\in\mathbb{N}$. If $k\geq 2$, then $xM\subseteq (P^2)_r$, and thus $xH_P=xM_P\subseteq (P_P^2)_{r_P}=x^2H_P$, a contradiction. Therefore, $((P^2)_r\cup xM)_r=P$. If $z\in H$ is such that $xz\in (P^2)_r$, then $xz\in ((P^2)_r)_P=x^2H_P$, and thus $z\in xH_P\cap H=P$. We infer that $x\in P\cap xH=(xM\cup (P^2)_r)_r\cap xH=(xM\cup ((P^2)_r\cap xH))_r\subseteq (xM\cup xP)_r=xM$, a contradiction.

(6) $\Rightarrow$ (7): It is clear that $H$ satisfies the Principal Ideal Theorem. It follows from \cite[Proposition 3.9]{R} that $H$ is primary $r$-ideal inclusive.

(7) $\Rightarrow$ (1): Recall that a prime $r$-ideal $P$ of $H$ is called $r$-branched if there exists a $P$-primary $r$-ideal $I$ of $H$ with $I\not=P$.

\smallskip
\textsc{Claim 2.} For each $r$-branched prime $r$-ideal $P$ of $H$, we have that $P\in\mathfrak X(H)$ and $H_P$ is a DVM.

\smallskip
Let $P$ be an $r$-branched prime $r$-ideal of $H$. Then $P_P$ is a principal ideal of $H_P$ by \cite[Proposition 5.2(1)]{R}. There is some $x\in H^{\bullet}$ such that $P_P=xH_P$. Observe that $P\in\mathcal{P}(xH)\subseteq\mathfrak X(H)$, and hence $P_P\in\mathfrak X(H_P)$. Therefore, every prime $s$-ideal of $H_P$ is principal, and hence $H_P$ is a DVM.\qed(Claim 2)

\smallskip
Let $M\in r$-$\max(H)$. It is sufficient to show that $M\in\mathfrak X(H)$ (then $M$ is $r$-branched, and hence $H_M$ is a DVM by Claim 2). Assume that $M\not\in\mathfrak X(H)$. Then there is some nontrivial prime $r$-ideal $P$ of $H$ such that $P\subsetneq M$. Since $H$ is primary $r$-ideal inclusive, we can find an $r$-branched prime $r$-ideal $Q$ of $H$ such that $P\subsetneq Q$. By Claim 2 we have that $Q\in\mathfrak X(H)$, a contradiction.
\end{proof}

\begin{lemma}\label{Lem 3.7} Let $r$ be a modular finitary ideal system on $H$. Then $H$ is primary $r$-ideal inclusive.
\end{lemma}

\begin{proof} Let $P,Q\in r$-${\rm spec}(H)$ be such that $P\subsetneq Q$. There exist $x\in Q\setminus P$ and $L\in\mathcal{P}((P\cup x^2H)_r)$ such that $L\subseteq Q$. Set $I=((P\cup x^2H)_r)_L\cap H$. Observe that $I$ is an $L$-primary $r$-ideal of $H$. It remains to show that $I\not=L$. Assume to the contrary that $I=L$.  Then $x\in ((P\cup x^2H)_r)_L=(P_L\cup x^2H_L)_{r_L}$. Since $r_L$ is a modular finitary ideal system on $H_L$, we obtain that $xH_L=(x^2H_L\cup P_L)_{r_L}\cap xH_L=(x^2H_L\cup (P_L\cap xH_L))_{r_L}=(x^2H_L\cup xP_L)_{r_L}=x(xH_L\cup P_L)_{r_L}\subseteq xL_L$. Therefore, $H_L\subseteq L_L$, a contradiction.
\end{proof}

Let $r$ be a finitary ideal system on $H$. Then $H$ is called an {\it $r$-Pr\"ufer monoid}, resp. an {\it $r$-B\'ezout monoid}, if every nontrivial $r$-finitely generated $r$-ideal of $H$ is $r$-invertible, resp. principal. Note that $H$ is an $r$-B\'ezout monoid if and only if $H$ is an $r$-Pr\"ufer monoid and $\mathcal{C}_r(H)$ is trivial. Note that $H$ is an $r$-Pr\"ufer monoid if and only if $H_M$ is a valuation monoid for all $M\in r$-$\max(H)$. In particular, if $H$ is an $r$-Pr\"ufer monoid, then $H$ is $r$-treed and $r$ is modular. Moreover, $H$ is an $s$-Pr\"ufer monoid if and only if $H$ is a valuation monoid.

\begin{corollary}\label{Cor 3.8} $[$cf. \cite{BP,G}$]$ Let $H\not=G$ and let $p$ and $r$ be finitary ideal systems on $H$ such that $p$ is modular and $p\leq r$. The following are equivalent:
\begin{itemize}
\item[(1)] $H$ is an $r$-almost Dedekind monoid.
\item[(2)] $H$ is an $\widetilde{r_p}$-almost Dedekind monoid.
\item[(3)] $\widetilde{r_p}$-$\max(H)=\mathfrak X(H)$ and $H$ satisfies the $\widetilde{r_p}$-prime power condition.
\item[(4)] $H$ satisfies the strong $\widetilde{r_p}$-prime power condition.
\item[(5)] $H$ satisfies the $\widetilde{r_p}$-prime power condition and the Principal Ideal Theorem.
\end{itemize}
If these equivalent conditions are satisfied, then $\widetilde{r_p}=r=t$.
\end{corollary}

\begin{proof} (1) $\Leftrightarrow$ (2): This is an immediate consequence of Lemma~\ref{Lem 2.1}(2).

(2) $\Leftrightarrow$ (3) $\Leftrightarrow$ (4) $\Leftrightarrow$ (5): This follows from Proposition~\ref{Prop 3.6} and Lemmas~\ref{Lem 2.1}(4) and~\ref{Lem 3.7}.

Now let the equivalent conditions be satisfied. Since $\widetilde{r_p}\leq r\leq t$, it is sufficient to show that every $\widetilde{r_p}$-ideal of $H$ is a $t$-ideal of $H$. Let $I\in\mathcal{I}_{\widetilde{r_p}}(H)$. Observe that $H$ is an $\widetilde{r_p}$-Pr\"ufer monoid, and hence every $\widetilde{r_p}$-finitely generated $\widetilde{r_p}$-ideal of $H$ is a $t$-ideal of $H$. Since $\widetilde{r_p}$ is finitary, we infer that $I$ is a directed union of $t$-ideals of $H$, and hence $I$ is a $t$-ideal of $H$.
\end{proof}

\begin{theorem}\label{Thm 3.9} Let $r$ be a finitary ideal system on $H$. The following are equivalent:
\begin{itemize}
\item[(1)] $H$ is an $r$-almost Dedekind $r$-SP-monoid.
\item[(2)] $H$ is $r$-treed and every nontrivial prime $r$-ideal of $H$ contains an $r$-invertible radical $r$-ideal of $H$.
\item[(3)] $H$ satisfies the $r$-prime power condition, $H$ is primary $r$-ideal inclusive and each nontrivial prime $r$-ideal of $H$ contains an $r$-invertible radical $r$-ideal.
\item[(4)] The radical of every nontrivial $r$-finitely generated $r$-ideal of $H$ is $r$-invertible.
\item[(5)] $\mathcal{P}(I)\subseteq\mathfrak X(H)$ for every nontrivial $r$-finitely generated $r$-ideal $I$ of $H$ and the radical of every nontrivial principal ideal of $H$ is $r$-invertible.
\end{itemize}
\end{theorem}

\begin{proof} Without restriction let $H\not=G$. (1) $\Rightarrow$ (2), (3): This follows from \cite[Corollary 3.4 and Propositions 3.9 and 3.10(1)]{R}.

(2) $\Rightarrow$ (1): This follows from Lemma~\ref{Lem 3.2}(2).

(3) $\Rightarrow$ (1): First we show that $H$ satisfies the Principal Ideal Theorem. Let $x\in H^{\bullet}$ and $P\in\mathcal{P}(xH)$. It follows by Claim 1 in the proof of Proposition~\ref{Prop 3.6} that $P_P$ is principal. Observe that every nontrivial prime $r_P$-ideal of $H_P$ contains a nontrivial radical principal ideal of $H_P$, and thus $P_P\in\mathfrak{X}(H_P)$ by \cite[Lemma 2.3(2)]{R} (since $P_P$ is principal and thus minimal above a nontrivial radical principal ideal of $H_P$). Therefore, $P\in\mathfrak{X}(H)$.

Consequently, $H$ is an $r$-almost Dedekind monoid by Proposition~\ref{Prop 3.6}. By \cite[Corollary 3.4]{R} we have that $H$ is an $r$-SP-monoid.

(1) $\Rightarrow$ (4): This follows from \cite[Corollary 3.4]{R}.

(4) $\Rightarrow$ (5): Let $I$ be a nontrivial $r$-finitely generated $r$-ideal of $H$. We infer by Proposition~\ref{Prop 3.1} that $\mathcal{P}(I)=\mathcal{P}(\sqrt{I})\subseteq\mathfrak X(H)$.

(5) $\Rightarrow$ (1): Let $M\in r$-$\max(H)$. Then $H_M$ is $r_M$-local, $r_M$ is a finitary ideal system on $H_M$ and the radical of every principal of $H_M$ is principal. Let $I$ be a nontrivial $r_M$-finitely generated $r_M$-ideal of $H_M$. Then $\mathcal{P}(I)\subseteq\mathfrak X(H_M)$. (Note that there is some nontrivial $r$-finitely generated $r$-ideal $J$ of $H$ such that $I=J_M$. Moreover, the contraction of every element of $\mathcal{P}(I)$ to $H$ is an element of $\mathcal{P}(J)$.) By Proposition~\ref{Prop 3.4}(1), there is some nonzero $z\in H_M$ such that $\mathcal{P}(I)=\{P\in\mathfrak X(H_M)\mid I\subseteq P\}=\{P\in\mathfrak X(H_M)\mid z\in P\}=\mathcal{P}(zH_M)$. This implies that $\sqrt[\uproot{3}H_M]{I}=\sqrt[\uproot{3}H_M]{zH_M}$ is principal. It follows by Proposition~\ref{Prop 3.5} that $H_M$ is a DVM. We infer that $H$ is an $r$-almost Dedekind monoid. It follows by \cite[Corollary 3.4]{R} that $H$ is an $r$-SP-monoid.
\end{proof}

\begin{theorem}\label{Thm 3.10} Let $r$ be a finitary ideal system on $H$. The following are equivalent:
\begin{itemize}
\item[(1)] $H$ is an $r$-B\'ezout $r$-SP-monoid.
\item[(2)] $H$ is a radical factorial $r$-B\'ezout monoid.
\item[(3)] $H$ is $r$-treed, $\mathcal{C}_r(H)$ is trivial and every nontrivial prime $r$-ideal of $H$ contains a nontrivial radical principal ideal of $H$.
\item[(4)] $H$ satisfies the $r$-prime power condition, $H$ is primary $r$-ideal inclusive and the radical of every principal ideal of $H$ is principal.
\item[(5)] $H$ is $r$-treed and the radical of every principal ideal of $H$ is principal.
\item[(6)] $\mathcal{P}(I)\subseteq\mathfrak X(H)$ for every nontrivial $r$-finitely generated $r$-ideal $I$ of $H$ and the radical of every principal ideal of $H$ is principal.
\item[(7)] The radical of every $r$-finitely generated $r$-ideal of $H$ is principal.
\end{itemize}
\end{theorem}

\begin{proof} (1) $\Rightarrow$ (2): Clearly, $\mathcal{C}_r(H)$ is trivial, and thus $H$ is radical factorial by \cite[Proposition 3.10(2)]{R}.

(2) $\Rightarrow$ (3): Since $H$ is an $r$-B\'ezout monoid, it is clear that $H$ is $r$-treed and $\mathcal{C}_r(H)$ is trivial. Since $H$ is radical factorial, every nontrivial prime $r$-ideal of $H$ contains a nontrivial radical principal ideal of $H$.

(3) $\Rightarrow$ (1): This is an immediate consequence of Theorem~\ref{Thm 3.9}, since every $r$-almost Dedekind monoid with trivial $r$-class group is an $r$-B\'ezout monoid.

(1) $\Rightarrow$ (4): It follows from Theorem~\ref{Thm 3.9} that $H$ satisfies the $r$-prime power condition, that $H$ is primary $r$-ideal inclusive and that the radical of every nontrivial principal ideal of $H$ is $r$-invertible. Since $H$ is an $r$-B\'ezout monoid, we infer that the radical of every principal ideal of $H$ is principal.

(4) $\Rightarrow$ (5): This is an immediate consequence of Theorem~\ref{Thm 3.9}.

(5) $\Rightarrow$ (6): Without restriction let $H\not=G$. It follows by Lemma~\ref{Lem 3.2}(2) and \cite[Proposition 2.10]{R} that $H$ is an $r$-almost Dedekind $r$-SP-monoid, and hence $r$-$\max(H)=\mathfrak X(H)$. Obviously, $\mathcal{P}(I)\subseteq\mathfrak X(H)$ for every nontrivial $r$-finitely generated $r$-ideal $I$ of $H$.

(6) $\Rightarrow$ (7): Let $I$ be a nontrivial $r$-finitely generated $r$-ideal of $H$. By Proposition~\ref{Prop 3.4}(1), we have that $\mathcal{P}(I)=\{P\in\mathfrak X(H)\mid I\subseteq P\}=\{P\in\mathfrak X(H)\mid z\in P\}=\mathcal{P}(zH)$ for some nonzero $z\in H$. Consequently, $\sqrt{I}=\sqrt{zH}$ is principal.

(7) $\Rightarrow$ (1): By Theorem~\ref{Thm 3.9}, $H$ is an $r$-almost Dedekind $r$-SP-monoid. We infer by Proposition~\ref{Prop 3.4}(2) that $H$ is an $r$-B\'ezout monoid.
\end{proof}

Next we rediscover several well-known characterizations for (B\'ezout) SP-domains and we also present some new characterizations.

\begin{corollary}\label{Cor 3.11} $[$cf. \cite[Lemma 4.2 and Theorem 4.3]{HOR} and \cite[Corollary 7.7]{OR}$]$ Let $R$ be an integral domain.
\begin{itemize}
\item[(A)] The following are equivalent:
\begin{itemize}
\item[(1)] $R$ is an SP-domain.
\item[(2)] $R$ is treed and every nonzero prime ideal of $R$ contains an invertible radical ideal of $R$.
\item[(3)] Every primary ideal of $R$ is a power of its radical and every nonzero prime ideal of $R$ contains an invertible radical ideal of $R$.
\item[(4)] Every minimal prime ideal of each nonzero finitely generated ideal of $R$ is of height one and the radical of every nonzero principal ideal of $R$ is invertible.
\item[(5)] The radical of every nonzero finitely generated ideal of $R$ is invertible.
\end{itemize}
\item[(B)] The following are equivalent:
\begin{itemize}
\item[(1)] $R$ is a B\'ezout SP-domain.
\item[(2)] $R$ is a radical factorial B\'ezout domain.
\item[(3)] $R$ is treed and the radical of every principal ideal of $R$ is principal.
\item[(4)] Every primary ideal of $R$ is a power of its radical and the radical of every principal ideal of $R$ is principal.
\item[(5)] Every minimal prime ideal of each nonzero finitely generated ideal of $R$ is of height one and the radical of every principal ideal of $R$ is principal.
\item[(6)] The radical of every finitely generated ideal of $R$ is principal.
\end{itemize}
\end{itemize}
\end{corollary}

\begin{proof} This is an easy consequence of Lemma~\ref{Lem 3.7} and Theorems~\ref{Thm 3.9} and~\ref{Thm 3.10}.
\end{proof}

Note that there are examples of $t$-SP-monoids that fail to be $t$-almost Dedekind monoids. As shown in \cite[Example 4.2]{RR} there is some $t$-local $t$-SP-monoid $H$ such that every nontrivial $t$-ideal of $H$ is $t$-cancellative and $t$-$\dim(H)=2$. In particular, $H$ satisfies the $t$-prime power condition and $\mathcal{P}(I)\subseteq\mathfrak{X}(H)$ for each nontrivial $t$-finitely generated $t$-ideal of $H$. Note that $H$ does not satisfy the strong $t$-prime power condition, $H$ is not $t$-treed and $H$ is not primary $t$-ideal inclusive.

\section{On the $t$-system and the $w$-system}

In this section we study the $t$-system and its modularizations. We present stronger characterizations for these types of finitary ideal systems than in the section before. Besides that, we investigate the connections with the modularizations $\widetilde{r_p}$ of a finitary ideal system $r$ in general and describe $\widetilde{r_p}$-SP-monoids and $\widetilde{r_p}$-B\'ezout $\widetilde{r_p}$-SP-monoids. We also show that the $t$-class group of every radical factorial BF-monoid is torsionfree. Let $r$ be a finitary ideal system on $H$. We say that $H$ is an {\it $r$-finite conductor monoid} if $xH\cap yH$ is $r$-finitely generated for all $x,y\in H$.

\begin{proposition}\label{Prop 4.1}$[$cf. \cite{GI,Z}$]$ Let $\mathcal{P}$ be a set of prime $s$-ideals of $H$ such that $\bigcap_{P\in\mathcal{P}} H_P=H$ and $H_Q$ is a valuation monoid for every $Q\in\mathcal{P}$. Let $I$ and $J$ be $t$-ideals of $H$.
\begin{itemize}
\item[(1)] If $I$, $J$ and $I\cap J$ are $t$-finitely generated, then $(IJ)_t=((I\cap J)(I\cup J)_t)_t$.
\item[(2)] If $I$ and $J$ are $t$-invertible and $I\cap J$ is $t$-finitely generated, then $I\cap J$ and $(I\cup J)_t$ are $t$-invertible.
\item[(3)] If $H$ is a $t$-finite conductor monoid, then $H$ is a $t$-Pr\"ufer monoid.
\end{itemize}
\end{proposition}

\begin{proof} Observe that $r:\mathbb{P}(H)\rightarrow\mathbb{P}(H)$ defined by $X_r=\bigcap_{P\in\mathcal{P}} (X_s)_P$ for each $X\subseteq H$ is an ideal system on $H$. This implies that $r\leq v$, and hence $I=\bigcap_{P\in\mathcal{P}} I_P$ for each divisorial ideal $I$ of $H$.

(1) Let $I$, $J$ and $I\cap J$ be $t$-finitely generated. Then $(IJ)_t$ and $(I\cup J)_t$ are $t$-finitely generated. This implies that $((I\cap J)(I\cup J))_t=((I\cap J)(I\cup J)_t)_t$ is $t$-finitely generated. Therefore, it is sufficient to show that $(((I\cap J)(I\cup J))_t)_P=((IJ)_t)_P$ for each $P\in\mathcal{P}$. Let $P\in\mathcal{P}$. Since $H_P$ is a valuation monoid, we have that $I_P\subseteq J_P$ or $J_P\subseteq I_P$. Consequently, $(((I\cap J)(I\cup J))_t)_P=((I_P\cap J_P)(I_P\cup J_P))_{t_P}=(I_PJ_P)_{t_P}=((IJ)_t)_P$.

(2) Let $I$ and $J$ be $t$-invertible and let $I\cap J$ be $t$-finitely generated. Clearly, $I$ and $J$ are $t$-finitely generated, and thus $((I\cap J)(I\cup J)_t)_t=(IJ)_t$. Since $(IJ)_t$ is t-invertible, we have that $((I\cap J)(I\cup J)_t)_t$ is $t$-invertible, and hence $I\cap J$ and $(I\cup J)_t$ are $t$-invertible.

(3) Let $H$ be a $t$-finite conductor monoid. First we show that for each nonempty finite $A\subseteq H$ and each $x\in H$ it follows that $A_t\cap xH$ is $t$-finitely generated. Let $A\subseteq H$ be finite and nonempty and $x\in H$. Let $P\in\mathcal{P}$. Since $H_P$ is a valuation monoid, we have that $(A_t)_P=AH_P$. We infer that $(A_t\cap xH)_P=AH_P\cap xH_P=\bigcup_{b\in A} (bH_P\cap xH_P)=(\bigcup_{b\in A} (bH_P\cap xH_P))_{t_P}=((\bigcup_{b\in A} (bH\cap xH))_t)_P$. This implies that $A_t\cap xH=(\bigcup_{b\in A} (bH\cap xH))_t$ is $t$-finitely generated. Next we show by induction that for each $n\in\mathbb{N}$ and all $E\subseteq H^{\bullet}$ with $|E|=n$ it follows that $E_t$ is $t$-invertible. The statement is clearly true for $n=1$. Now let $n\in\mathbb{N}$ and $F\subseteq H^{\bullet}$ be such that $|F|=n+1$. There exist $E\subseteq F$ and $x\in F\setminus E$ such that $F=E\cup\{x\}$ and $|E|=n$. It follows by the previous claim that $E_t\cap xH$ is $t$-finitely generated. We infer by (2) that $F_t=(E_t\cup xH)_t$ is $t$-invertible.
\end{proof}

\begin{theorem}\label{Thm 4.2} The following are equivalent:
\begin{itemize}
\item[(1)] $H$ is a $t$-almost Dedekind $t$-SP-monoid.
\item[(2)] $H$ is a $t$-finite conductor monoid and every principal ideal of $H$ is a finite $t$-product of radical $t$-ideals of $H$.
\item[(3)] Every $t$-ideal of $H$ is a $t$-product of finitely many pairwise comparable radical $t$-ideals of $H$.
\item[(4)] The radical of every nontrivial principal ideal of $H$ is $t$-invertible.
\end{itemize}
\end{theorem}

\begin{proof} (1) $\Rightarrow$ (2): This is obvious.

(2) $\Rightarrow$ (1): By Proposition~\ref{Prop 3.1} we have that $\bigcap_{P\in\mathfrak X(H)} H_P=H$ and $H_Q$ is a DVM for every $Q\in\mathfrak X(H)$. It follows by Proposition~\ref{Prop 4.1}(3) that $H$ is a $t$-Pr\"ufer monoid, and hence $H$ is $t$-treed. Consequently, $H$ is a $t$-almost Dedekind $t$-SP-monoid by Theorem~\ref{Thm 3.9}.

(1) $\Rightarrow$ (3): This follows from \cite[Theorem 3.3(2)]{R}.

(3) $\Rightarrow$ (4): Let $x\in H^{\bullet}$. There exist $n\in\mathbb{N}$ and finitely many radical $t$-ideals $I_i$ of $H$ such that $I_i\subseteq I_{i+1}$ for each $i\in [1,n-1]$ and $xH=(\prod_{i=1}^n I_i)_t$. This implies that $\sqrt{xH}=\bigcap_{i=1}^n I_i=I_1$ is $t$-invertible.

(4) $\Rightarrow$ (1): By Theorem~\ref{Thm 3.9} it is sufficient to show that the radical of every nontrivial $t$-finitely generated $t$-ideal of $H$ is $t$-invertible.

It follows by Proposition~\ref{Prop 3.1} that $\bigcap_{P\in\mathfrak X(H)} H_P=H$, $H_Q$ is a DVM for each $Q\in\mathfrak X(H)$ and $\mathcal{P}(A)\subseteq\mathfrak X(H)$ for each $t$-invertible $t$-ideal $A$ of $H$.

It is sufficient to show by induction that for each $n\in\mathbb{N}$ and each $E\subseteq H^{\bullet}$ with $|E|=n$ it follows that $\sqrt{E_t}$ is $t$-invertible. The statement is clearly true for $n=1$. Now let $n\in\mathbb{N}$ and $F\subseteq H^{\bullet}$ be such that $|F|=n+1$. There exist $E\subseteq F$ and $x\in F\setminus E$ such that $|E|=n$ and $F=E\cup\{x\}$. Set $I=\sqrt{E_t}$ and $J=\sqrt{xH}$. Then $I$ and $J$ are $t$-invertible radical $t$-ideals of $H$. Observe that $\sqrt{F_t}=\sqrt{(I\cup J)_t}$, since the radical of every $t$-ideal of $H$ is a $t$-ideal of $H$. Moreover, $I\cap J=\sqrt{(xE)_t}$ is $t$-invertible, since $|xE|=|E|=n$. We infer by Proposition~\ref{Prop 4.1}(2) that $(I\cup J)_t$ is $t$-invertible. Note that

\begin{align*}
\sqrt{(I\cup J)_t}&=\bigcap_{P\in\mathfrak X(H), (I\cup J)_t\subseteq P} P=\Bigg(\bigcap_{P\in\mathfrak X(H), (I\cup J)_t\subseteq P} P_P\Bigg)\cap H\\
&=\Bigg(\bigcap_{P\in\mathfrak X(H), (I\cup J)_t\subseteq P} P_P\Bigg)\cap\Bigg(\bigcap_{P\in\mathfrak X(H), (I\cup J)_t\not\subseteq P} H_P\Bigg)\\
&=\Bigg(\bigcap_{P\in\mathfrak X(H), (I\cup J)_t\subseteq P} ((I\cup J)_t)_P\Bigg)\cap\Bigg(\bigcap_{P\in\mathfrak X(H), (I\cup J)_t\not\subseteq P} ((I\cup J)_t)_P\Bigg)\\
&=\bigcap_{P\in\mathfrak X(H)} ((I\cup J)_t)_P=(I\cup J)_t,
\end{align*}

\noindent where the first equality holds since $\mathcal{P}((I\cup J)_t)\subseteq\mathfrak X(H)$, and the last equality holds since $(I\cup J)_t$ is $t$-finitely generated (and hence divisorial). Therefore, $\sqrt{F_t}=(I\cup J)_t$ is $t$-invertible.
\end{proof}

\begin{theorem}\label{Thm 4.3} Let $H\not=G$ and let $p$ and $r$ be finitary ideal systems on $H$ such that $p$ is modular and $p\leq r$.
\begin{itemize}
\item[(A)] The following are equivalent:
\begin{itemize}
\item[(1)] $H$ is an $r$-almost Dedekind $r$-SP-monoid.
\item[(2)] $r$-$\max(H)=t$-$\max(H)$ and the radical of every nontrivial principal ideal of $H$ is $t$-invertible.
\item[(3)] $H$ is an $\widetilde{r_p}$-SP-monoid.
\end{itemize}
\item[(B)] The following are equivalent:
\begin{itemize}
\item[(1)] $H$ is an $r$-B\'ezout $r$-SP-monoid.
\item[(2)] $r$-$\max(H)=t$-$\max(H)$ and the radical of every principal ideal of $H$ is principal.
\item[(3)] $H$ is an $\widetilde{r_p}$-B\'ezout $\widetilde{r_p}$-SP-monoid.
\end{itemize}
\end{itemize}
\end{theorem}

\begin{proof} (A) (1) $\Rightarrow$ (2): First let $H$ be an $r$-almost Dedekind $r$-SP-monoid. Clearly, $r$-$\max(H)=\mathfrak X(H)$, and since every height-one prime $s$-ideal of $H$ is a $t$-ideal, we infer that $r$-$\max(H)=t$-$\max(H)$. By Theorem~\ref{Thm 3.9}, the radical of every nontrivial principal ideal of $H$ is $r$-invertible. Since $r\leq t$, we have that the radical of every nontrivial principal ideal of $H$ is $t$-invertible.

(2) $\Rightarrow$ (1): Now let $r$-$\max(H)=t$-$\max(H)$ and let the radical of every nontrivial principal ideal of $H$ be $t$-invertible. It follows by Theorem~\ref{Thm 4.2} that $H$ is a $t$-almost Dedekind monoid, and hence $r$-$\max(H)=t$-$\max(H)=\mathfrak X(H)$. Therefore, $H$ is $r$-treed and every $t$-invertible $t$-ideal of $H$ is an $r$-invertible $r$-ideal of $H$. Consequently, $H$ is an $r$-almost Dedekind $r$-SP-monoid by Theorem~\ref{Thm 3.9}.

(2) $\Leftrightarrow$ (3): By Lemmas~\ref{Lem 2.1}(4) and~\ref{Lem 3.7}, Theorem~\ref{Thm 3.9} and \cite[Proposition 3.10(1)]{R}, we have that $H$ is an $\widetilde{r_p}$-SP-monoid if and only if $H$ is an $\widetilde{r_p}$-almost Dedekind $\widetilde{r_p}$-SP-monoid. Now applying the equivalence of (1) and (2) to $\widetilde{r_p}$ and using the fact that $r$-$\max(H)=\widetilde{r_p}$-$\max(H)$ gives us the desired equivalence.

\smallskip
(B) This is an easy consequence of (A), Proposition~\ref{Prop 3.4}(2) and Theorem~\ref{Thm 3.10}.
\end{proof}

\begin{corollary}\label{Cor 4.4} The following are equivalent:
\begin{itemize}
\item[(1)] $H$ is a $t$-almost Dedekind $t$-SP-monoid.
\item[(2)] $H$ is a $w$-SP-monoid.
\item[(3)] $H$ is a $w$-finite conductor monoid and every principal ideal of $H$ is a finite $w$-product of radical $w$-ideals of $H$.
\item[(4)] Every $w$-ideal of $H$ is a $w$-product of finitely many pairwise comparable radical $w$-ideals of $H$.
\item[(5)] The radical of every nontrivial principal ideal of $H$ is $w$-invertible.
\end{itemize}
\end{corollary}

\begin{proof} (1) $\Rightarrow$ (2): By Theorem~\ref{Thm 4.2}, the radical of every nontrivial principal ideal is $t$-invertible. As pointed out before, we have that $w$-$\max(H)=t$-$\max(H)$. We infer by Theorem~\ref{Thm 4.3}(A) that $H$ is a $w$-almost Dedekind $w$-SP-monoid.

(2) $\Rightarrow$ (3): This is obvious, since $H$ is a $w$-almost Dedekind monoid.

(3) $\Rightarrow$ (1): Let $x,y\in H$. Then $xH\cap yH=E_w$ for some finite $E\subseteq H$. Since $w\leq t$, we infer that $xH\cap yH=(xH\cap yH)_t=(E_w)_t=E_t$. Therefore, $H$ is a $t$-finite conductor monoid. Note that every nontrivial principal ideal of $H$ is a finite $w$-product of $w$-invertible radical $w$-ideals of $H$. Therefore, every nontrivial principal ideal of $H$ is a finite $t$-product of ($t$-invertible) radical $t$-ideals of $H$ by Lemma~\ref{Lem 2.1}(3). The statement now follows from Theorem~\ref{Thm 4.2}.

(2) $\Rightarrow$ (4) $\Rightarrow$ (5): This can be proved along the same lines as in Theorem~\ref{Thm 4.2}.

(5) $\Rightarrow$ (1): Since every $w$-invertible $w$-ideal of $H$ is a $t$-invertible $t$-ideal of $H$, the radical of every nontrivial principal ideal of $H$ is $t$-invertible. Therefore, $H$ is a $t$-almost Dedekind $t$-SP-monoid by Theorem~\ref{Thm 4.2}.
\end{proof}

\begin{corollary}\label{Cor 4.5} The following are equivalent:
\begin{itemize}
\item[(1)] $H$ is a $t$-B\'ezout $t$-SP-monoid.
\item[(2)] $H$ is a $w$-B\'ezout $w$-SP-monoid.
\item[(3)] The radical of every principal ideal of $H$ is principal.
\item[(4)] Every principal ideal of $H$ is a product of finitely many pairwise comparable radical principal ideals.
\end{itemize}
\end{corollary}

\begin{proof} (1) $\Leftrightarrow$ (2) $\Leftrightarrow$ (3): This follows from Theorem~\ref{Thm 4.3}(B).

(3) $\Rightarrow$ (4): It follows by \cite[Lemma 2.3(2)]{R} that $H$ satisfies the Principal Ideal Theorem. Let $x\in H$ be nonzero. Clearly, there is a sequence $(z_i)_{i\in\mathbb{N}}$ of nonzero radical elements of $H$ such that $\sqrt{(x/\prod_{i=1}^{\ell-1} z_i)H}=z_{\ell}H$ for each $\ell\in\mathbb{N}$. Moreover, we have that $z_{\ell}H\subseteq z_{\ell+1}H$ for all $\ell\in\mathbb{N}$. Since $\sqrt{xH}=z_1H$, there is some $k\in\mathbb{N}$ such that $z_1^k\in xH$. We infer by Lemma~\ref{Lem 3.3} that $z_{k+1}\in H^{\times}$, and thus $xH=\prod_{i=1}^k z_iH$.

(4) $\Rightarrow$ (3): Let $x\in H^{\bullet}$. Then there exist $n\in\mathbb{N}$ and finitely many radical principal ideals $I_i$ of $H$ such that $xH=\prod_{i=1}^n I_i$ and $I_i\subseteq I_{i+1}$ for all $i\in [1,n-1]$. It follows that $\sqrt{xH}=\bigcap_{i=1}^n I_i=I_1$ is principal.
\end{proof}

Note that $w$ can be replaced by $w_p$ in Corollaries~\ref{Cor 4.4} and~\ref{Cor 4.5}, where $p$ is an arbitrary modular finitary ideal system on $H$.

\begin{corollary}\label{Cor 4.6} $H$ is factorial if and only if the radical of every principal ideal of $H$ is principal and $H$ satisfies the ascending chain condition on radical principal ideals.
\end{corollary}

\begin{proof} This is an immediate consequence of Theorem~\ref{Thm 3.10}, Corollary~\ref{Cor 4.5} and \cite[Theorem 2.14]{R}.
\end{proof}

Finally, we give a partial positive answer to the (so far) unsolved problem of whether the $t$-class group of a radical factorial monoid is torsionfree. The following result shows that the $t$-class group of a radical factorial monoid has to satisfy a ``weak form'' of being torsionfree. Let $H$ be a monoid and $\mathcal{A}\subseteq\mathbb{P}(H)$. A function $\lambda:\mathcal{A}\rightarrow\mathbb{N}_0$ is called a length function on $\mathcal{A}$ if $\lambda(J)<\lambda(I)$ for all $I,J\in\mathcal{A}$ with $I\subsetneq J$. Moreover, $H$ is called a {\it BF-monoid} if the set of nontrivial principal ideals of $H$ possesses a length function.

\begin{proposition}\label{Prop 4.7} Let $H$ be a radical factorial monoid, $k\in\mathbb{N}$, $I\in\mathcal{I}_t^*(H)$ such that $(I^k)_t$ is principal and $\mathcal{A}=\{(L^k)_t\mid L\in\mathcal{I}_t^*(H),I\subseteq L,(L^k)_t$ is principal$\}$.
\begin{enumerate}
\item[(1)] If $\mathcal{A}$ possesses a length function, then $I$ is principal.
\item[(2)] If $\{P\in\mathfrak{X}(H)\mid I\subseteq P\}$ is finite, then $I$ is principal.
\item[(3)] If $H$ is a BF-monoid, then $\mathcal{C}_t(H)$ is torsionfree.
\end{enumerate}
\end{proposition}

\begin{proof} (1) Let $\lambda:\mathcal{A}\rightarrow\mathbb{N}_0$ be a length function on $\mathcal{A}$. It is sufficient to show by induction that for each $n\in\mathbb{N}_0$ and $L\in\mathcal{I}_t^*(H)$ such that $I\subseteq L$, $(L^k)_t$ is principal and $\lambda((L^k)_t)=n$, it follows that $L$ is principal. Let $n\in\mathbb{N}_0$ and $L\in\mathcal{I}_t^*(H)$ be such that $I\subseteq L$, $(L^k)_t$ is principal and $\lambda((L^k)_t)=n$. Without restriction let $L\not=H$. There is some radical nonunit $x\in H$ such that $L\subseteq\sqrt{L}=\sqrt{(L^k)_t}\subseteq xH$. Consequently, $L=xJ$ for some $J\in\mathcal{I}_t^*(H)$. Note that $(J^k)_t$ is principal, $I\subseteq J$ and $(L^k)_t\subsetneq (J^k)_t$. We infer that $\lambda((J^k)_t)<n$, and hence $J$ is principal by the induction hypothesis. This implies that $L$ is principal.

(2) Let $\{P\in\mathfrak{X}(H)\mid I\subseteq P\}$ be finite. Let $\mathcal{P}$ be the set of all finite $t$-products (which are not necessarily squarefree or nonempty) of elements of $\mathfrak{X}(H)$. Since $H_Q$ is a DVM for each $Q\in\mathfrak{X}(H)$ by Proposition~\ref{Prop 3.1}, we infer that $\{C\in\mathcal{P}\mid (I^k)_t\subseteq C\}$ is finite. Let $\lambda:\mathcal{A}\rightarrow\mathbb{N}_0$ be defined by $\lambda(L)=|\{C\in\mathcal{P}\mid L\subseteq C\}|$ for each $L\in\mathcal{A}$. Now let $A,B\in\mathcal{A}$ be such that $A\subsetneq B$. There exist $x,y\in H$ and some nonunit $z\in H$ such that $A=xH$, $B=yH$ and $x=yz$. Since $H$ satisfies the Principal Ideal Theorem by Proposition~\ref{Prop 3.1}, there is some $Q\in\mathfrak{X}(H)$ such that $z\in Q$. Moreover, there is some minimal $J\in\mathcal{P}$ such that $yH\subseteq J$. We have that $A=xH\subseteq (JQ)_t\in\mathcal{P}$ and $B=yH\nsubseteq (JQ)_t\subsetneq J$ (note that $H_Q$ is a DVM). Therefore, $\lambda(B)<\lambda(A)$, and thus $\lambda$ is a length function. The statement now follows by (1).

(3) Let $H$ be a BF-monoid, $\ell\in\mathbb{N}$ and $L\in\mathcal{I}_t^*(H)$ such that $(L^{\ell})_t$ is principal. Set $\mathcal{B}=\{(J^{\ell})_t\mid J\in\mathcal{I}_t^*(H),L\subseteq J,(J^{\ell})_t$ is principal$\}$. Since $\mathcal{B}$ is a subset of the set of nontrivial principal ideals of $H$, we have that $\mathcal{B}$ possesses a length function. Therefore, $L$ is principal by (1). We infer that $\mathcal{C}_t(H)$ is torsionfree.
\end{proof}

\section{On the monoid of $r$-invertible $r$-ideals}

In this section, we put our focus on the monoid of $r$-invertible $r$-ideals and give characterizations for this monoid to be radical factorial or to have the property that the radical of every principal ideal is principal. We also present a characterization for radical factorial monoids and discuss the connections between the monoid of $r$-invertible $r$-ideals and radical $r$-factorization of principal ideals and $r$-invertible $r$-ideals.

\begin{lemma}\label{Lem 5.1} Let $r$ be a finitary ideal system on $H$ and $I,J\in\mathcal{I}_r^*(H)$.
\begin{itemize}
\item[(1)] $I$ divides $J$ in $\mathcal{I}_r^*(H)$ if and only if $J\subseteq I$.
\item[(2)] $I$ is radical if and only if $I$ is a radical element of $\mathcal{I}_r^*(H)$.
\end{itemize}
\end{lemma}

\begin{proof} (1) Let $J$ be an $r$-invertible $r$-ideal of $H$. If $I$ divides $J$ in $\mathcal{I}_r^*(H)$, then $J=(IA)_r$ for some $r$-invertible $r$-ideal $A$ of $H$, and thus $J=(IA)_r\subseteq (IH)_r=I$. Conversely, if $J\subseteq I$, then $B=(JI^{-1})_r$ is an $r$-invertible $r$-ideal of $H$ and $J=(BI)_r$, and hence $I$ divides $J$ in $\mathcal{I}_r^*(H)$.

(2) First let $I$ be radical, $J\in\mathcal{I}_r^*(H)$ and $k\in\mathbb{N}$ such that $I$ divides $(J^k)_r$ in $\mathcal{I}_r^*(H)$. We have that $(J^k)_r\subseteq I$ by (1), and hence $J\subseteq\sqrt{J}=\sqrt{(J^k)_r}\subseteq I$. Therefore, $I$ divides $J$ in $\mathcal{I}_r^*(H)$ by (1).

Conversely, let $I$ be a radical element of $\mathcal{I}_r^*(H)$. Let $x\in\sqrt{I}$ be nonzero. There is some $k\in\mathbb{N}$ such that $(xH)^k\subseteq I$. Then $I$ divides $(xH)^k$ in $\mathcal{I}_r^*(H)$ by (1), and thus $I$ divides $xH$ in $\mathcal{I}_r^*(H)$. We infer that $x\in xH\subseteq I$ by (1).
\end{proof}

Let $r$ be a finitary ideal system on $H$. Next we present some (technical) characterizations of radical factorial monoids and monoids whose $r$-invertible $r$-ideals are finite $r$-products of radical $r$-ideals. Let $\Omega$ be a finite set of $r$-ideals of $H$ and $I$ an $r$-ideal of $H$. For each $P\in\mathfrak{X}(H)$ let $k_P$ be the number of elements of $\Omega$ which are contained in $P$. Then $\Omega$ is called {\it $(r,I)$-meager} if for each $P\in\mathfrak{X}(H)$ we have that $I\subseteq (P^{k_P})_r$.

\begin{proposition}\label{Prop 5.2} Let $r$ be a finitary ideal system on $H$.
\begin{itemize}
\item[(A)] The following are equivalent:
\begin{itemize}
\item[(1)] $\mathcal{I}_r^*(H)$ is radical factorial.
\item[(2)] Each $I\in\mathcal{I}_r^*(H)$ is a finite $r$-product of radical $r$-ideals of $H$.
\item[(3)] $\bigcap_{P\in\mathfrak X(H)} H_P=H$, $H_Q$ is a DVM for all $Q\in\mathfrak X(H)$ and for each $I\in\mathcal{I}_r^*(H)$, $\sqrt{I}=\bigcap_{J\in\Omega} J$ for some $(r,I)$-meager set $\Omega\subseteq\mathcal{I}_r^*(H)$.
\end{itemize}
\item[(B)] $H$ is radical factorial if and only if $\bigcap_{P\in\mathfrak X(H)} H_P=H$, $H_Q$ is a DVM for each $Q\in\mathfrak X(H)$ and for each $x\in H$, $\sqrt{xH}=\bigcap_{J\in\Omega} J$ for some $(t,xH)$-meager set $\Omega$ of principal ideals of $H$.
\end{itemize}
\end{proposition}

\begin{proof} Observe that if $\bigcap_{P\in\mathfrak X(H)} H_P=H$, then $g:\mathbb{P}(H)\rightarrow\mathbb{P}(H)$ defined by $X_g=\bigcap_{P\in\mathfrak X(H)} (X_s)_P$ for each $X\subseteq H$ is an ideal system on $H$. In particular, if $\bigcap_{P\in\mathfrak X(H)} H_P=H$, then $I=\bigcap_{P\in\mathfrak X(H)} I_P$ for each divisorial ideal $I$ of $H$.

(A) (1) $\Leftrightarrow$ (2): Let $I$ be an $r$-invertible $r$-ideal of $H$. By Lemma~\ref{Lem 5.1}(1) we have that $I$ is a finite $r$-product of radical $r$-ideals of $H$ if and only if $I$ is a finite $r$-product of $r$-invertible radical $r$-ideals of $H$ if and only if $I$ is a finite product of radical elements of $\mathcal{I}_r^*(H)$. Now the statement follows easily.

(2) $\Rightarrow$ (3): We infer by Proposition~\ref{Prop 3.1} that $\bigcap_{P\in\mathfrak X(H)} H_P=H$ and $H_Q$ is a DVM for all $Q\in\mathfrak X(H)$. Let $I$ be an $r$-invertible $r$-ideal of $H$. Then $I=(\prod_{i=1}^n I_i)_r$ for some $n\in\mathbb{N}$ and finitely many radical $r$-ideals $I_i$ of $H$. Set $\Omega=\{I_i\mid i\in [1,n]\}$. Clearly, $\Omega$ is a finite set of $r$-invertible $r$-ideals of $H$ and $\sqrt{I}=\bigcap_{i=1}^n I_i=\bigcap_{J\in\Omega} J$. Let $P\in\mathfrak X(H)$ and set $k=|\{J\in\Omega\mid J\subseteq P\}|$. Then $I\subseteq (\prod_{J\in\Omega} J)_r\subseteq (P^k)_r$, and hence $\Omega$ is an $(r,I)$-meager set.

\smallskip
(3) $\Rightarrow$ (2): \textsc{Claim.} For each $r$-invertible $r$-ideal $I$ of $H$ there is some $m\in\mathbb{N}_0$ such that $I\nsubseteq (P^m)_r$ for all $P\in\mathfrak{X}(H)$.

\smallskip
Let $I$ be an $r$-invertible $r$-ideal of $H$. Then there is some $(r,I)$-meager set $\Omega$ of $r$-invertible $r$-ideals of $H$ such that $\sqrt{I}=\bigcap_{J\in\Omega} J$. First we show that each element of $\Omega$ is a radical $r$-ideal of $H$. Let $J\in\Omega$ and $P\in\mathfrak X(H)$.

Case 1: $\sqrt{J}\nsubseteq P$. We have that $J\nsubseteq P$, and hence $(\sqrt{J})_P=J_P$.

Case 2: $\sqrt{J}\subseteq P$. Then $\sqrt{I}\subseteq P$, and thus $P_P=(\sqrt{I})_P\subseteq J_P\subseteq (\sqrt{J})_P\subseteq P_P$. Therefore, $(\sqrt{J})_P=J_P$.

Since $J$ is $r$-invertible, $J$ is divisorial, and since $\sqrt{J}$ is an intersection of $r$-invertible $r$-ideals of $H$, $\sqrt{J}$ is divisorial. Consequently, $\sqrt{J}=\bigcap_{Q\in\mathfrak{X}(H)} (\sqrt{J})_Q=\bigcap_{Q\in\mathfrak{X}(H)} J_Q=J$.

Note that $\sqrt{I}=\sqrt{\bigcap_{J\in\Omega} J}=\sqrt{(\prod_{J\in\Omega} J)_r}$, and since $(\prod_{J\in\Omega} J)_r$ is $r$-finitely generated, there is some $k\in\mathbb{N}$ such that $(\prod_{J\in\Omega} J^k)_r\subseteq I$. Set $\ell=|\{J\in\Omega\mid J\subseteq P\}|$ and $m=1+k|\Omega|$. Assume that $I\subseteq (P^m)_r$. Then $(\prod_{J\in\Omega} J^k)_r\subseteq (P^m)_r$, and hence $P_P^{k\ell}=((\prod_{J\in\Omega} J^k)_r)_P\subseteq ((P^m)_r)_P=P_P^m$. Since $k\ell<m$ this contradicts the fact that $P_P$ is a nontrivial proper principal ideal of $H_P$.\qed(Claim)

\smallskip
For $A\in\mathcal{I}_r^*(H)$, we set $m_A=\max\{k\in\mathbb{N}_0\mid A\subseteq (P^k)_r$ for some $P\in\mathfrak{X}(H)\}$ (which exists by the claim). It is sufficient to show by induction that for all $m\in\mathbb{N}_0$ and $I\in\mathcal{I}_r^*(H)$ with $m_I=m$, that $I$ is a finite $r$-product of radical $r$-ideals of $H$.

Let $m\in\mathbb{N}_0$ and $I\in\mathcal{I}_r^*(H)$ be such that $m_I=m$. If $m=0$, then since $I$ is divisorial, $I=\bigcap_{P\in\mathfrak X(H)} I_P=\bigcap_{P\in\mathfrak X(H)} H_P=H$ and we are done. Now let $m>0$. There is some $(r,I)$-meager set $\Omega\subseteq\mathcal{I}_r^*(H)$ such that $\sqrt{I}=\bigcap_{J\in\Omega} J$. As in the proof of the claim, it follows that each element of $\Omega$ is a radical $r$-ideal of $H$.

Let $P\in\mathfrak X(H)$ and set $\ell=|\{J\in\Omega\mid J\subseteq P\}|$. Then $((\prod_{J\in\Omega} J)_r)_P=\prod_{J\in\Omega} J_P=P_P^{\ell}=((P^{\ell})_r)_P\supseteq I_P$. Since $I$ and $(\prod_{J\in\Omega} J)_r$ are divisorial, we have that $I\subseteq (\prod_{J\in\Omega} J)_r$. We infer that $I=(L\prod_{J\in\Omega} J)_r$ for some $r$-invertible $r$-ideal $L$ of $H$. It is sufficient to show that $m_L<m$. Without restriction let $m_L>0$. There is some $Q\in\mathfrak{X}(H)$ such that $m_L=\max\{k\in\mathbb{N}_0\mid L\subseteq (Q^k)_r\}$. Since $m_L>0$, we have that $I\subseteq L\subseteq Q$, and thus $J\subseteq Q$ for some $J\in\Omega$. Since $I\subseteq (JL)_r\subseteq (Q^{m_L+1})_r$, we infer that $m_L<m_L+1\leq m$.

\smallskip
(B) This can be shown along the same lines as ``(A) (2) $\Leftrightarrow$ (A) (3)'', by replacing $r$ with $t$ and by replacing $r$-invertible $r$-ideals with nontrivial principal ideals.
\end{proof}

\begin{corollary}\label{Cor 5.3} Let $r$ be a finitary ideal system on $H$. Then $H$ is an $r$-almost Dedekind $r$-SP-monoid if and only if $H$ is an $r$-Pr\"ufer monoid and $\mathcal{I}_r^*(H)$ is radical factorial.
\end{corollary}

\begin{proof} Note that every $r$-almost Dedekind monoid is an $r$-Pr\"ufer monoid and every $r$-Pr\"ufer monoid is $r$-treed. Moreover, if every $r$-invertible $r$-ideal of $H$ is a finite $r$-product of radical $r$-ideals of $H$, then clearly every nontrivial prime $r$-ideal of $H$ contains an $r$-invertible radical $r$-ideal of $H$. Therefore, the equivalence is an immediate consequence of Theorem~\ref{Thm 3.9} and Proposition~\ref{Prop 5.2}(A).
\end{proof}

\begin{proposition}\label{Prop 5.4} Let $r$ be a finitary ideal system on $H$. The following are equivalent:
\begin{itemize}
\item[(1)] Every principal ideal of $H$ is an $r$-product of finitely many pairwise comparable radical $r$-ideals of $H$.
\item[(2)] The radical of every nontrivial principal ideal of $H$ is $r$-invertible.
\item[(3)] The radical of every $r$-invertible $r$-ideal of $H$ is $r$-invertible.
\item[(4)] The radical of every principal ideal of $\mathcal{I}_r^*(H)$ is principal.
\item[(5)] Every $r$-invertible $r$-ideal of $H$ is an $r$-product of finitely many pairwise comparable radical $r$-ideals of $H$.
\end{itemize}
\end{proposition}

\begin{proof} (1) $\Rightarrow$ (2): This is straightforward to prove.

(2) $\Rightarrow$ (3): Recall that a nontrivial $r$-ideal $J$ of $H$ is $r$-invertible if and only if $J_t$ is $t$-finitely generated and $J_M$ is principal for each $M\in r$-$\max(H)$. Let $I$ be an $r$-invertible $r$-ideal of $H$. We have to show that $\sqrt{I}$ is $t$-finitely generated and $(\sqrt{I})_M$ is principal for each $M\in r$-$\max(H)$. Clearly, the radical of every nontrivial principal ideal of $H$ is $t$-invertible, and hence $\sqrt{I}$ is $t$-invertible by Theorem~\ref{Thm 4.2}. Therefore, $\sqrt{I}$ is $t$-finitely generated. Let $M\in r$-$\max(H)$. Observe that the radical of every principal ideal of $H_M$ is principal, and thus $(\sqrt{I})_M=\sqrt[\uproot{3}H_M]{I_M}$ is principal.

(3) $\Rightarrow$ (4): Let $I$ be an $r$-invertible $r$-ideal of $H$. Set $J=\sqrt{I}$ and $\mathcal{I}=\mathcal{I}_r^*(H)$. It is sufficient to show that $\sqrt[\uproot{3}\mathcal{I}]{I\mathcal{I}}=J\mathcal{I}$. Since $I\subseteq J$, we infer by Lemma~\ref{Lem 5.1}(1) that $J$ divides $I$ in $\mathcal{I}$, and hence $I\mathcal{I}\subseteq J\mathcal{I}$. Since $J$ is a radical element of $\mathcal{I}$ by Lemma~\ref{Lem 5.1}(2), we have that $\sqrt[\uproot{3}\mathcal{I}]{I\mathcal{I}}\subseteq J\mathcal{I}$. Since $J$ is $r$-finitely generated, there is some $n\in\mathbb{N}$ such that $(J^n)_r\subseteq I$. Therefore, $I$ divides $(J^n)_r$ in $\mathcal{I}$ by Lemma~\ref{Lem 5.1}(1), and hence $(J\mathcal{I})^n=(J^n)_r\mathcal{I}\subseteq I\mathcal{I}$. This implies that $J\mathcal{I}\subseteq\sqrt[\uproot{3}\mathcal{I}]{I\mathcal{I}}$.

(4) $\Rightarrow$ (5): Let $I$ be an $r$-invertible $r$-ideal of $H$. Set $\mathcal{I}=\mathcal{I}_r^*(H)$. By Corollary~\ref{Cor 4.5}, there exist $n\in\mathbb{N}$ and finitely many radical elements $I_i$ of $\mathcal{I}$ such that $I\mathcal{I}=\prod_{i=1}^n I_i\mathcal{I}=(\prod_{i=1}^n I_i)_r\mathcal{I}$ and $I_i\mathcal{I}\subseteq I_{i+1}\mathcal{I}$ for all $i\in [1,n-1]$. This implies that $I=(\prod_{i=1}^n I_i)_r$. Let $i\in [1,n]$. It follows by Lemma~\ref{Lem 5.1}(2) that $I_i$ is a radical $r$-ideal of $H$. Furthermore, if $i\in [1,n-1]$, then $I_i\subseteq I_{i+1}$ by Lemma~\ref{Lem 5.1}(1).

(5) $\Rightarrow$ (1): This is obvious.
\end{proof}

\section{Monoid rings and $*$-Nagata rings}

\begin{center}
\textit{In this section let $H$ always be a monoid with $z(H)=\emptyset$.}
\end{center}

As an application, we study several ring-theoretical constructions in this section. Recall that the monoid $H$ is {\it completely integrally closed} if for all $x\in H$ and $y\in G$ with $xy^n\in H$ for all $n\in\mathbb{N}$, it follows that $y\in H$. Moreover, $H$ is called {\it root-closed} if for all $x\in G$ and $n\in\mathbb{N}$ with $x^n\in H$, we have that $x\in H$. We say that $H$ is a {\it grading monoid} if $H$ is torsionless (i.e., for all $x,y\in H$ and $n\in\mathbb{N}$ such that $x^n=y^n$ it follows that $x=y$). If not stated otherwise, we will write a grading monoid additively (from now on). Note that $H$ is a grading monoid if and only if we can define a total order on it which is compatible to the monoid operation (\cite[page 123]{N}). Moreover, a nontrivial Abelian group is a grading monoid if and only if it is torsionfree. Let $R$ be an integral domain, $H$ a grading monoid, $K$ be a field of quotients of $R$ and $G$ a quotient group of $H$. A sequence $(x_g)_{g\in I}$ of elements of $K$ is called {\it formally infinite} if all but finitely many elements of that sequence are zero.

\smallskip
By $R[H]=R[X;H]=\{\sum_{g\in H} x_gX^g\mid (x_g)_{g\in H}\in R^H$ is formally infinite$\}$ we denote the monoid ring over $R$ and $H$. It is well-known that $R[H]$ is an integral domain. Note that $R[H]$ is integrally closed if and only if $R$ is integrally closed and $H$ is root-closed (\cite[Theorem 3.7(d)]{DFA}). Furthermore, $R[H]$ is completely integrally closed if and only if $R$ and $H$ are completely integrally closed (\cite[Theorem 3.7(e)]{DFA}). If $B\subseteq K$ and $Y\subseteq G$, then set $B[Y]=\{\sum_{g\in Y} x_gX^g\mid (x_g)_{g\in Y}\in B^Y$ is formally infinite$\}$. Let $S=\{yX^g\mid y\in R\setminus\{0\},g\in H\}$ denote the set of nonzero homogeneous elements of $R[H]$. Then $S^{-1}(R[H])=K[G]$ is called the homogeneous field of quotients of $R[H]$. It is well-known that $K[G]$ is a completely integrally closed $t$-B\'ezout domain (\cite[Theorem 2.2]{DFA}). An ideal $A$ of $R[H]$ is called homogeneous if for all formally infinite $(x_g)_{g\in H}\in R^H$ such that $\sum_{g\in H} x_gX^g\in A$ we have that $x_gX^g\in A$ for all $g\in H$ (equivalently, $A$ is generated by homogeneous elements of $R[H]$). Let $I$ be an ideal of $R$ and let $Y$ be an $s$-ideal of $H$. Then $I[Y]$ is a homogeneous ideal of $R[H]$. Also note that if $J$ is an ideal of $R$ and $Z$ is an $s$-ideal of $H$, then $I[Y]J[Z]=(IJ)[Y+Z]$.

\smallskip
Finally, note that if $R$ is an integral domain, then the $t$-system on $R$ and the ``classical'' $t$-operation on $R$ coincide for nonzero ideals of $R$. More precisely, the $t$-system on $R$ extends the $t$-operation on $R$ to arbitrary subsets of $R$. For this reason, we do not have to distinguish between the ring theoretical and the monoid theoretical definition of ``$t$'' on integral domains. Since the monoid ring $R[H]$ is an integral domain (if $H$ is a (torsionless) grading monoid), these considerations also apply to $R[H]$.

\begin{lemma}\label{Lem 6.1} Let $R$ be an integral domain, $H$ a grading monoid, $I$ an ideal of $R$, $Y$ an $s$-ideal of $H$ and $A$ a nonzero ideal of $R[H]$.
\begin{itemize}
\item[(1)] $\sqrt[{R[H]}]{I[Y]}=\sqrt[R]{I}[\sqrt[H]{Y}]$.
\item[(2)] $(I[Y])_{t_{R[H]}}=I_{t_R}[Y_{t_H}]$.
\item[(3)] Let $I$ be a $t$-ideal of $R$ and $Y$ a $t$-ideal of $H$. Then $I[Y]$ is $t$-invertible if and only if $I$ and $Y$ are $t$-invertible.
\item[(4)] $A=J[Z]$ for some $t$-ideal $J$ of $R$ and some $t$-ideal $Z$ of $H$ if and only if $A$ is a homogeneous $t$-ideal of $R[H]$ if and only if $A=F_t$ for some nonempty set $F$ of nonzero homogeneous elements of $R[H]$.
\item[(5)] If $R[H]$ is integrally closed and $A$ is a $t$-ideal of $R[H]$ that contains a nonzero homogeneous element of $R[H]$, then $A$ is homogeneous.
\end{itemize}
\end{lemma}

\begin{proof} (1) Recall that there is some total order $\leq$ on $H$ that is compatible with the monoid operation on $H$.

First let $f\in\sqrt{I[Y]}$ be nonzero. Then $f^k\in I[Y]\subseteq\sqrt{I}[\sqrt{Y}]$ for some $k\in\mathbb{N}$. We have that $f=\sum_{i=1}^n f_iX^{a_i}$ for some $n\in\mathbb{N}$, $(f_i)_{i=1}^n\in (R^{\bullet})^n$ and $(a_i)_{i=1}^n\in H^n$ with $a_j<a_k$ for all $j,k\in [1,n]$ with $j<k$. We show by induction on $m$ that $\sum_{i=1}^m f_iX^{a_i}\in\sqrt{I}[\sqrt{Y}]$ for all $m\in [1,n]$. Let $m\in [1,n]$. Set $g=\sum_{i=1}^{m-1} f_iX^{a_i}$. Then $g\in\sqrt{I}[\sqrt{Y}]$ by the induction hypothesis. We have that $\sum_{a\in H} h_aX^a=(f-g)^k=\sum_{i=0}^k(-1)^{k-i}\binom{k}{i}f^ig^{k-i}\in\sqrt{I}[\sqrt{Y}]$ for some formally infinite $(h_a)_{a\in H}\in R^H$. Note that $h_{ka_m}=f_m^k$. Therefore, $f_m^kX^{ka_m}\in\sqrt{I}[\sqrt{Y}]$. Since $f_m^k\not=0$, we infer that $f_m^k\in\sqrt{I}$ and $ka_m\in\sqrt{Y}$. Therefore, $f_m\in\sqrt{I}$ and $a_m\in\sqrt{Y}$, and thus $f_mX^{a_m}\in\sqrt{I}[\sqrt{Y}]$. This implies that $\sum_{i=1}^m f_iX^{a_i}=g+f_mX^{a_m}\in\sqrt{I}[\sqrt{Y}]$.

Conversely, let $f\in\sqrt{I}[\sqrt{Y}]$ be nonzero. Then $f=\sum_{i=1}^n f_iX^{a_i}$ for some $n\in\mathbb{N}$, $(f_i)_{i=1}^n\in (R^{\bullet})^n$ and $(a_i)_{i=1}^n\in H^n$ with $a_j<a_k$ for all $j,k\in [1,n]$ with $j<k$. This implies that $f_i\in\sqrt{I}$ and $a_i\in\sqrt{Y}$ for each $i\in [1,n]$. Consequently, there is some $m\in\mathbb{N}$ such that $f_i^m\in I$ and $ma_i\in Y$ for each $i\in [1,n]$. Let $(m_i)_{i=1}^n\in\mathbb{N}_0^n$ be such that $\sum_{i=1}^n m_i=mn$. Clearly, there is some $j\in [1,n]$ such that $m_j\geq m$. We have that $\prod_{i=1}^n (f_iX^{a_i})^{m_i}=f_j^{m_j}X^{m_ja_j}\prod_{i=1,i\not=j}^n (f_iX^{a_i})^{m_i}\in I[Y]$. Note that $f^{mn}$ is a sum of elements of the form $\prod_{i=1}^n (f_iX^{a_i})^{m_i}$ with $m_i\in\mathbb{N}_0$ and $\sum_{i=1}^n m_i=mn$. Therefore, $f^{mn}\in I[Y]$, and hence $f\in\sqrt{I[Y]}$.

(2), (3) This follows from \cite[Corollary 2.4]{BIK}.

(4) Let $S$ denote the set of nonzero homogeneous elements of $R[H]$. We only need to show that if $A=F_t$ for some nonempty $F\subseteq S$, then $A=J[Z]$ for some $t$-ideal $J$ of $R$ and some $t$-ideal $Z$ of $H$. Set $T=\{E\subseteq S\mid\emptyset\not=E\subseteq A,|E|<\infty\}$. Observe that $A=\bigcup_{E\in T} E_v$. Let $D\in T$. By \cite[Proposition 2.5]{AA} we have that $D_v$ is a homogeneous divisorial ideal of $R[H]$. It follows from \cite[Proposition 2.5]{BIK} that there exist an ideal $J_D$ of $R$ and an $s$-ideal $Z_D$ of $H$ such that $D_v=J_D[Z_D]$. Therefore, for each $C\in T$, there exist an ideal $J_C$ of $R$ and an $s$-ideal $Z_C$ of $H$ such that $C_v=J_C[Z_C]$. Set $J=\bigcup_{C\in T} J_C$ and $Y=\bigcup_{C\in T} Y_C$. Note that if $B,C\in T$ are such that $B\subseteq C$, then $J_B[Z_B]=B_v\subseteq C_v=J_C[Z_C]$, and hence $J_B\subseteq J_C$ and $Z_B\subseteq Z_C$ (since $B_v\not=\{0\}$). Consequently, $J$ is an ideal of $R$ and $Y$ is an $s$-ideal of $H$. Moreover, $A=\bigcup_{E\in T} J_E[Z_E]=J[Z]$. (Note that if $x\in J[Z]$, then $x$ can be represented as a finite sum of elements of the form $x_bX^b$ with $x_b\in J$ and $b\in Z$, and hence there is some $E\in T$ such that all homogeneous components of $x$ are in $J_E[Z_E]$.) We infer that $J[Z]=A=A_t=J_t[Z_t]$, and thus $J_t=J$ is a $t$-ideal of $R$ and $Y_t=Y$ is a $t$-ideal of $H$.

(5) Let $R[H]$ be integrally closed and $A$ a $t$-ideal of $R[H]$ that contains a nonzero homogeneous element $x\in R[H]$. Let $f\in A$. Then there is some finite $E\subseteq A$ such that $\{x,f\}\subseteq E_v$. It follows from \cite[Theorems 3.2 and 3.7]{AF} that $E_v$ is homogeneous. Therefore, all homogeneous components of $f$ are contained in $E_v\subseteq A$.
\end{proof}

\begin{proposition}\label{Prop 6.2} Let $K$ be a field and $G$ a nontrivial torsionfree Abelian group. The following are equivalent:
\begin{itemize}
\item[(1)] The radical of every principal ideal of $K[G]$ is principal.
\item[(2)] $K[G]$ is radical factorial.
\item[(3)] $K[G]$ is a $t$-SP-domain.
\item[(4)] Every nonzero prime $t$-ideal of $K[G]$ contains a nonzero radical principal ideal of $K[G]$.
\end{itemize}
If $G$ satisfies the ascending chain condition on cyclic subgroups, then these equivalent conditions are satisfied.
\end{proposition}

\begin{proof} The equivalence is an immediate consequence of Theorem~\ref{Thm 3.10} and Corollary~\ref{Cor 4.5}. Now let $G$ satisfy the ascending chain condition on cyclic subgroup. It follows from \cite[Theorem 2.3(a)]{DFA} that $K[G]$ is factorial, and thus $K[G]$ is radical factorial.
\end{proof}

Note that the equivalent conditions in Proposition~\ref{Prop 6.2} are not always satisfied. Let $p$ be a prime number, $G$ a nontrivial additive torsionfree $p$-divisible Abelian group (e.g. $(G,+)=(\mathbb{Q},+)$ or $(G,+)=(\mathbb{Z}[\frac{1}{p}],+)$) and $K$ a field of characteristic $p$. Then $K[G]$ does not satisfy the equivalent conditions in Proposition~\ref{Prop 6.2}. Assume to the contrary that the radical of every principal ideal of $K[G]$ is principal. Let $a\in G$ be nonzero. There is some $f\in K[G]$ such that $\sqrt{(1+X^a)K[G]}=fK[G]$. Consequently, there is some $m\in\mathbb{N}$ such that $f^{p^m}\in (1+X^a)K[G]$. There is some nonzero $b\in G$ such that $p^mb=a$. Observe that $K[G]$ has also characteristic $p$, and hence $1+X^a=(1+X^b)^{p^m}$. Note that $\frac{f}{1+X^b}$ is an element of the field of quotients of $K[G]$. Since $K[G]$ is completely integrally closed, and thus root-closed, we infer that $f\in (1+X^b)K[G]$. It follows that $\sqrt{(1+X^b)K[G]}=\sqrt{(1+X^a)K[G]}=fK[G]=(1+X^b)K[G]$. There is some nonzero $c\in G$ such that $pc=b$. Consequently, $1+X^c\in\sqrt{(1+X^b)K[G]}=(1+X^c)^pK[G]$, and thus $1+X^c\in K[G]^{\times}$, a contradiction.

\begin{proposition}\label{Prop 6.3} Let $R$ be an integral domain, $H$ a grading monoid and $S$ the set of nonzero homogeneous elements of $R[H]$. The following are equivalent:
\begin{itemize}
\item[(1)] $R[H]$ is integrally closed and every $t$-ideal $A$ of $R[H]$ with $A\cap S\not=\emptyset$ is a finite $t$-product of radical $t$-ideals of $R[H]$.
\item[(2)] $R[H]$ is integrally closed and every homogeneous $t$-ideal of $R[H]$ is a finite $t$-product of radical $t$-ideals of $R[H]$.
\item[(3)] Every $t$-ideal $A$ of $R[H]$ with $A\cap S\not=\emptyset$ is a finite $t$-product of homogeneous radical $t$-ideals of $R[H]$.
\item[(4)] Every homogeneous $t$-ideal of $R[H]$ is a finite $t$-product of homogeneous radical $t$-ideals of $R[H]$.
\item[(5)] $R$ is a $t$-SP-domain and $H$ is a $t$-SP-monoid.
\end{itemize}
\end{proposition}

\begin{proof} (1) $\Rightarrow$ (2), (3) $\Rightarrow$ (4): This is obviously true.

(1) $\Rightarrow$ (3), (2) $\Rightarrow$ (4): This is an immediate consequence of Lemma~\ref{Lem 6.1}(5).

(4) $\Rightarrow$ (5): Let $I$ be a nonzero $t$-ideal of $R$ and let $Y$ be a nonempty $t$-ideal of $H$. Then $I[Y]$ is a homogeneous $t$-ideal of $R[H]$ by Lemma~\ref{Lem 6.1}(4). Therefore, there exist $n\in\mathbb{N}$ and finitely many homogeneous radical $t$-ideals $A_i$ of $R[H]$ such that $I[Y]=(\prod_{i=1}^n A_i)_t$. It follows from Lemma~\ref{Lem 6.1} that, for each $j\in [1,n]$, there is some radical $t$-ideal $I_j$ of $R$ and some radical $t$-ideal $Y_j$ of $H$ such that $A_j=I_j[Y_j]$. We infer that $I[Y]=(\prod_{i=1}^n I_i[Y_i])_t=((\prod_{i=1}^n I_i)[\sum_{i=1}^n Y_i])_t=(\prod_{i=1}^n I_i)_t[(\sum_{i=1}^n Y_i)_t]$, and hence $I=(\prod_{i=1}^n I_i)_t$ and $Y=(\sum_{i=1}^n Y_i)_t$.

(5) $\Rightarrow$ (1): It follows from \cite[Proposition 3.10(3)]{R} that $R$ and $H$ are completely integrally closed. Therefore, $R[H]$ is completely integrally closed, and hence it is integrally closed.

Now let $A$ be a nonzero $t$-ideal of $R[H]$ such that $A\cap S\not=\emptyset$. By Lemma~\ref{Lem 6.1} there exist a $t$-ideal $I$ of $R$ and a $t$-ideal $Y$ of $H$ such that $A=I[Y]$. There exist $n,m\in\mathbb{N}$, finitely many radical $t$-ideals $I_i$ of $R$ such that $I=(\prod_{i=1}^n I_i)_t$ and finitely many radical $t$-ideals $Y_j$ of $H$ such that $Y=(\sum_{j=1}^m Y_j)_t$. We infer by Lemma~\ref{Lem 6.1} that $I_i[H]$ is a homogeneous radical $t$-ideal of $R[H]$ for all $i\in [1,n]$ and $R[Y_j]$ is a homogeneous radical $t$-ideal of $R[H]$ for all $j\in [1,m]$. Finally, we have that $A=(I[H]R[Y])_t=((\prod_{i=1}^n I_i)_t[H]R[(\sum_{j=1}^m Y_j)_t])_t=((\prod_{i=1}^n I_i[H])_t(\prod_{j=1}^m R[Y_j])_t)_t=(\prod_{i=1}^n I_i[H]\prod_{j=1}^m R[Y_j])_t$.
\end{proof}

\begin{proposition}\label{Prop 6.4} $[$cf. \cite[Proposition 2.17]{R}$]$ Let $R$ be an integral domain, $H$ a grading monoid, $K$ a field of quotients of $R$ and $G$ a quotient group of $H$.
\begin{itemize}
\item[(1)] $R[H]$ is a $w$-SP-domain if and only if $R$ is a $w$-SP-domain, $H$ is a $w$-SP-monoid and $K[G]$ is radical factorial.
\item[(2)] $R[H]$ is a $w$-B\'ezout $w$-SP-domain if and only if $R$ is a $w$-B\'ezout $w$-SP-domain, $H$ is a  $w$-B\'ezout $w$-SP-monoid and $K[G]$ is radical factorial.
\end{itemize}
\end{proposition}

\begin{proof} (1) Let $R$ be a $w$-SP-domain, let $H$ be a $w$-SP-monoid and let $K[G]$ be radical factorial. Note that $R$ is a $t$-Pr\"ufer domain (i.e., a P$v$MD) and $H$ is a $t$-Pr\"ufer monoid by Theorem~\ref{Thm 4.2} and Corollaries~\ref{Cor 4.4} and~\ref{Cor 5.3}. Therefore, $R[H]$ is a $t$-Pr\"ufer domain by \cite[Proposition 6.5]{AF}. In particular, if $A$ and $B$ are $t$-invertible $t$-ideals of $R[H]$, then $(AB)_t=((A\cap B)(A\cup B))_t$, and hence $A\cap B$ is $t$-invertible. For $g\in R[H]$ let $C(g)$ be the ideal of $R[H]$ generated by the homogeneous components of $g$. Since $R$ and $H$ are completely integrally closed, $R[H]$ is completely integrally closed, and hence it follows by \cite[Lemmas 1.5 and 1.6]{BIK} that for every nonzero $g\in R[H]$, $gC(g)^{-1}=gK[G]\cap R[H]$. In particular, if $g\in R[H]$ is nonzero, then $gR[H]=g(C(g)_tC(g)^{-1})_t=(C(g)_tgC(g)^{-1})_t=(C(g)_t(gK[G]\cap R[H]))_t$ (since $R[H]$ is a $t$-Pr\"ufer domain) and $gK[G]\cap R[H]$ is a $t$-invertible $t$-ideal of $R[H]$.

By Theorem~\ref{Thm 4.2} and Corollary~\ref{Cor 4.4} it is sufficient to show that the radical of every nonzero principal ideal of $R[H]$ is $t$-invertible. Let $f\in R[H]$ be nonzero. It follows from Proposition~\ref{Prop 6.2} that there is some $g\in R[H]$ such that $\sqrt{fK[G]\cap R[H]}=(\sqrt[{K[G]}]{fK[G]})\cap R[H]=gK[G]\cap R[H]$ (here we use that $K[G]$ is a quotient overring of $R[H]$). This implies that $\sqrt{fK[G]\cap R[H]}$ is $t$-invertible. Moreover, $C(f)_t=I[Y]$ for some $t$-invertible $t$-ideal $I$ of $R$ and some $t$-invertible $t$-ideal $Y$ of $H$ by Lemma~\ref{Lem 6.1}. It follows from Theorem~\ref{Thm 4.2} and Corollary~\ref{Cor 4.4} that $\sqrt{I}$ is a $t$-invertible $t$-ideal of $R$ and $\sqrt{Y}$ is a $t$-invertible $t$-ideal of $H$. Therefore, $\sqrt{C(f)_t}=\sqrt{I}[\sqrt{Y}]$ is $t$-invertible by Lemma~\ref{Lem 6.1}.

We have that $fR[H]=(C(f)_t(fK[G]\cap R[H]))_t$, and thus $\sqrt{fR[H]}=\sqrt{C(f)_t}\cap\sqrt{fK[G]\cap R[H]}$ is $t$-invertible.

\smallskip
Now let $R[H]$ be a $w$-SP-domain. First let $y\in R$ be nonzero. We have that $\sqrt{yR}[H]=\sqrt{(yR)[H]}=\sqrt{yR[H]}$ is $t$-invertible by Lemma~\ref{Lem 6.1}(1), Theorem~\ref{Thm 4.2} and Corollary~\ref{Cor 4.4}. Consequently, $\sqrt{yR}$ is $t$-invertible by Lemma~\ref{Lem 6.1}(3). It follows by Theorem~\ref{Thm 4.2} and Corollary~\ref{Cor 4.4} that $R$ is a $w$-SP-domain. Now let $z\in H$. It follows that $R[\sqrt{z+H}]=\sqrt{R[z+H]}=\sqrt{X^zR[H]}$ is $t$-invertible by Lemma~\ref{Lem 6.1}(1), Theorem~\ref{Thm 4.2} and Corollary~\ref{Cor 4.4}. Therefore, $\sqrt{z+H}$ is $t$-invertible by Lemma~\ref{Lem 6.1}(3). We infer again by Theorem~\ref{Thm 4.2} and Corollary~\ref{Cor 4.4} that $H$ is a $w$-SP-monoid. Finally, let $f\in K[G]$ be nonzero. Let $S$ be the set of nonzero homogeneous elements of $R[H]$. There is some nonzero $g\in R[H]$ such that $fK[G]=gK[G]$. By Theorem~\ref{Thm 4.2} and Corollary~\ref{Cor 4.4} we have that $\sqrt{gR[H]}$ is $t$-invertible, and hence $\sqrt[{K[G]}]{fK[G]}=S^{-1}\sqrt[{R[H]}]{gR[H]}$ is an $S^{-1}t$-invertible $S^{-1}t$-ideal of $K[G]=S^{-1}(R[H])$. Since $S^{-1}t\leq t_{K[G]}$, this implies that $\sqrt[{K[G]}]{fK[G]}$ is a $t$-invertible $t$-ideal of $K[G]$. Consequently, $\sqrt[{K[G]}]{fK[G]}$ is a principal ideal of $K[G]$, since $K[G]$ is a $t$-B\'ezout domain.

\smallskip
(2) In any case, $R[H]$ is completely integrally closed by \cite[Proposition 3.10]{R}, and hence $\mathcal{C}_t(R[H])\cong\mathcal{C}_t(R)\bigoplus\mathcal{C}_t(H)$ by \cite[Corollary 2.11]{BIK}. In particular, $\mathcal{C}_t(R[H])$ is trivial if and only if $\mathcal{C}_t(R)$ and $\mathcal{C}_t(H)$ are both trivial. Therefore, the statement follows by (1) and Theorems~\ref{Thm 3.9} and~\ref{Thm 3.10}.
\end{proof}

Let $R$ be an integral domain with quotient field $K$ and $H$ a grading monoid with quotient group $G$. Note that if $R$ is a $t$-SP-domain, $H$ is a $t$-SP-monoid and $K[G]$ is factorial, then $R[H]$ is in general not a $t$-SP-domain.

Let $K$ be a field, $G=\mathbb{Z}^{(\mathbb{N}_0)}$ (i.e., $G$ is isomorphic to the free Abelian group with basis $\mathbb{N}_0$) and $H=\{(x_j)_{j\in\mathbb{N}_0}\in G\mid x_0\geq x_i\geq 0$ for all $i\in\mathbb{N}_0\}$. Note that $G$ is isomorphic to the direct sum of countably many copies of $\mathbb{Z}$. Clearly, $H$ is a grading monoid. It follows from \cite[Example 4.2]{RR} that $H$ is a $t$-SP-monoid and $\mathcal{C}_t(H)$ is trivial (since $H$ is $t$-local). Let $(X_i)_{i\in\mathbb{N}_0}$ be a sequence of independent indeterminates over $K$. Set $T=K[\{\prod_{i=0}^{\infty} X_i^{\alpha_i}\mid (\alpha_i)_{i\in\mathbb{N}_0}\in H\}]$ and $S=K[\{X_i,X_i^{-1}\mid i\in\mathbb{N}_0\}]$. It is clear that $K[H]\cong T$, $T$ is a subring of $K[\{X_i\mid i\in\mathbb{N}_0\}]$ and $K[G]\cong S$ is factorial. First we show that $T$ is not radical factorial. Let $f=X_0^3(X_1+1)^2(X_2^3+X_1)$. Then $f\in T^{\bullet}\setminus T^{\times}$. It is sufficient to show that $f$ is an atom of $T$ that is not radical. Since $K[X_1]$ is factorial, it follows by Eisenstein's criterion that $X_2^3+X_1$ is a prime element of $K[X_1,X_2]$. Therefore, $X_2^3+X_1$ is a prime element of $K[\{X_i\mid i\in\mathbb{N}_0\}]$. It is clear that $X_0$ and $X_1+1$ are prime elements of $K[\{X_i\mid i\in\mathbb{N}_0\}]$. Let $g,h\in T$ be such that $f=gh$. Since $K[\{X_i\mid i\in\mathbb{N}_0\}]$ is factorial, there are $\eta\in K^{\times}$, $a\in\{0,1,2,3\}$, $b\in\{0,1,2\}$ and $c\in\{0,1\}$ such that $g=\eta X_0^a(X_1+1)^b(X_2^3+X_1)^c$ and $h=\eta^{-1}X_0^{3-a}(X_1+1)^{2-b}(X_2^3+X_1)^{1-c}$. Without restriction let $c=1$. Since $g\in T$, we infer that $a=3$, and thus $b=2$ (since $h\in T$). This implies that $h=\eta^{-1}\in K^{\times}=T^{\times}$, and hence $f$ is an atom of $T$. Note that $X_1+1$ and $X_2^3+X_1$ are prime elements of $S$. Since $S$ is factorial and $f$ is not a square-free product of prime elements of $S$, we have that $f$ is not a radical element of $S$. Since $S$ is a quotient overring of $T$, we infer that $f$ is not a radical element of $T$. Consequently, $T$ is not radical factorial. Since $H$ is completely integrally closed, it follows by \cite[Lemma 2.1 and Corollary 2.10]{BIK} that $\mathcal{C}_t(T)\cong\mathcal{C}_t(H)$, and thus $\mathcal{C}_t(T)$ is trivial. Therefore, if $T$ is a $t$-SP-domain, then $T$ is radical factorial, a contradiction.

Next we provide a simple way to construct nontrivial examples of $w$-SP-monoids (or $t$-SP-monoids) that are grading monoids (if nontrivial examples of $w$-SP-domains or $t$-SP-domains are already given). Note that if $H$ is root-closed, then $H$ is a grading monoid if and only if $H^{\times}$ is torsionfree. (If $H$ is a grading monoid, then $H$ is torsionless, and thus $H^{\times}$ is torsionfree. Now let $H$ be root-closed and let $H^{\times}$ be torsionfree. Let $n\in\mathbb{N}$ and $x,y\in H$ be such that $nx=ny$. Then $n(x-y)=0\in H$, and thus $x-y\in H$, since $H$ is root-closed. We infer that $x-y\in H^{\times}$. Since $H^{\times}$ is torsionfree and $n(x-y)=0$, we have that $x=y$. Therefore, $H$ is a grading monoid.)

\begin{remark}\label{Rem 6.5} Let $R$ be an integral domain, $H$ a monoid and $U$ a subgroup of $H^{\times}$ with $U\not=H$. Set $H/U=\{xU\mid x\in H\}$ and let $V$ be a subgroup of $R^{\times}$ such that $R^{\times}/V$ is torsionfree (e.g. $V=R^{\times}$ or $V=\{x\in R\mid x^n=1$ for some $n\in\mathbb{N}\}$) and $V\not=R^{\bullet}$.
\begin{enumerate}
\item[(1)] $R$ is a $w$-SP-domain (resp. a $t$-SP-domain) if and only if $R^{\bullet}$ is a $w$-SP-monoid (resp. a $t$-SP-monoid).
\item[(2)] $H$ is a $w$-SP-monoid (resp. a $t$-SP-monoid) if and only if $H/U$ is a $w$-SP-monoid (resp. a $t$-SP-monoid).
\item[(3)] $R$ is a $w$-SP-domain (resp. a $t$-SP-domain) if and only if $R^{\bullet}/V$ is a $w$-SP-monoid (resp. a $t$-SP-monoid). If these equivalent conditions are satisfied, then $R^{\bullet}/V$ is a grading monoid.
\end{enumerate}
\end{remark}

\begin{proof} (1) Note that $f:\mathcal{I}_t(R)\rightarrow\mathcal{I}_t(R^{\bullet})$ defined by $f(I)=I\setminus\{0\}$ for each $I\in\mathcal{I}_t(R)$ is a semigroup isomorphism. Moreover, if $I\in\mathcal{I}_t(R)$, then $I$ is radical if and only if $f(I)$ is radical. Therefore, the statement is an immediate consequence of Theorem~\ref{Thm 4.2} and Corollary~\ref{Cor 4.4}.

(2) Observe that $f:\mathcal{I}_t(H)\rightarrow\mathcal{I}_t(H/U)$ defined by $f(I)=\{xU\mid x\in I\}$ for each $I\in\mathcal{I}_t(H)$ is a semigroup isomorphism. Furthermore, if $I\in\mathcal{I}_t(H)$, then $I$ is radical if and only if $f(I)$ is radical. Again, the statement is a consequence of Theorem~\ref{Thm 4.2} and Corollary~\ref{Cor 4.4}.

(3) The first statement follows from (1) and (2). Set $A=R^{\bullet}/V$ and suppose $A$ is a $t$-SP-monoid. Clearly, $A$ is a root-closed monoid whose elements are cancellative. Observe that $A^{\times}=R^{\times}/V$ is torsionfree. Therefore, $A$ is a grading monoid.
\end{proof}

Let $R$ be an integral domain and $X$ an indeterminate over $R$. We say that $\ast$ is a star operation on $R$ if $\ast$ is an ideal system on $R$ such that $d\leq\ast$. Moreover, we say that $\ast$ is a star operation of finite type if $\ast$ is a finitary ideal system on $R$. Let $\ast$ be a star operation of finite type on $R$. We say that $R$ is a P$\ast$MD if $R$ is a $\ast$-Pr\"ufer domain. For $f\in R[X]$ let $c(f)$ be the content of $f$. Set $N_{\ast}=\{g\in R[X]\mid c(g)_{\ast}=R\}$. By ${\rm Na}(R,\ast)=\{\frac{f}{g}\mid f\in R[X], g\in N_{\ast}\}$ we denote the $\ast$-Nagata ring of $R$.

\begin{proposition}\label{Prop 6.6} Let $R$ be an integral domain, $\ast$ a star operation of finite type on $R$ and $X$ an indeterminate over $R$. Then $R$ is a $\ast$-almost Dedekind $\ast$-SP-domain if and only if ${\rm Na}(R,\ast)$ is an SP-domain.
\end{proposition}

\begin{proof} Set $S={\rm Na}(R,\ast)$. Let $\mathcal{I}$ denote the monoid of $\ast$-invertible $\ast$-ideals of $R$ and let $\mathcal{J}$ denote the monoid of invertible ideals of $S$. It follows from Corollary~\ref{Cor 5.3} that $R$ is a $\ast$-almost Dedekind $\ast$-SP-domain if and only if $R$ is a P$\ast$MD and $\mathcal{I}$ is radical factorial. Since every SP-domain is an almost Dedekind domain, it follows by analogy that $S$ is an SP-domain if and only if $S$ is a Pr\"ufer domain and $\mathcal{J}$ is radical factorial. We infer by \cite[Theorem 3.1]{FJS} that $R$ is a P$\ast$MD if and only if $S$ is a Pr\"ufer domain. Therefore, it is sufficient to show that if $R$ is a P$\ast$MD, then the map $\varphi:\mathcal{I}\rightarrow\mathcal{J}$ defined by $\varphi(I)=IS$ for all $I\in\mathcal{I}$ is a monoid isomorphism. Let $R$ be a P$\ast$MD. It follows by \cite[Lemma 2.4]{FJS} that $\varphi$ is a well-defined map. We continue by showing the following claim.

\textsc{Claim.} If $I$ is a nonzero finitely generated ideal of $R$, then $I_{\ast}S=IS$.

Let $I$ be a nonzero finitely generated ideal of $R$. We have to show that $I_{\ast}\subseteq IS$. Let $x\in I_{\ast}$. Note that $IS=\{\frac{f}{g}\mid f\in I[X], g\in N_{\ast}\}$. Since $I_{\ast}$ is $\ast$-invertible, we have that $(II^{-1})_{\ast}=R$, and hence there is some finite $E\subseteq II^{-1}\subseteq R$ such that $E_{\ast}=R$. Clearly, there is some $g\in R[X]$ such that $c(g)={}_R(E)$. Observe that $c(g)_{\ast}=E_{\ast}=R$, and thus $g\in N_{\ast}$. Moreover, $Ex\subseteq II^{-1}I_{\ast}=I(I_{\ast})^{-1}I_{\ast}\subseteq I$, and hence $Ex\subseteq I$. Consequently, $gx\in I[X]$. It follows that $x\in IS$.\qed(Claim)

\smallskip
Now let $A$ and $B$ be $\ast$-invertible $\ast$-ideals of $R$. There exist nonzero finitely generated ideals $I$ and $J$ of $R$ such that $A=I_{\ast}$ and $B=J_{\ast}$. We infer by the claim that $\varphi((AB)_{\ast})=(AB)_{\ast}S=(IJ)_{\ast}S=IJS=ISJS=I_{\ast}SJ_{\ast}S=ASBS=\varphi(A)\varphi(B)$. Since $\varphi(R)=RS=S$, it follows that $\varphi$ is a monoid homomorphism.

To show that $\varphi$ is injective, it is sufficient to show that $AS\cap R=A$ for all $\ast$-invertible $\ast$-ideals $A$ of $R$. Let $A$ be a $\ast$-invertible $\ast$-ideal of $R$ and $x\in AS\cap R$. There is some $g\in N_{\ast}$ such that $gx\in A[X]$, and thus $c(g)x\subseteq A$. This implies that $x\in xR=xc(g)_{\ast}=(xc(g))_{\ast}\subseteq A_{\ast}=A$.

Finally, we show that $\varphi$ is surjective. By \cite[Lemma 2.5 and Remark 3.1]{FJS} we have that $S$ is a B\'ezout domain. Therefore, we need to show that for each nonzero $f\in R[X]$, there is some $\ast$-invertible $\ast$-ideal $A$ of $R$ such that $\varphi(A)=fS$. Let $f\in R[X]$ be nonzero. Set $A=c(f)_{\ast}$. Then $A$ is a $\ast$-invertible $\ast$-ideal of $R$ and it follows by \cite[Lemma 2.5 and Remark 3.1]{FJS} and the claim that $\varphi(A)=AS=c(f)S=fS$.
\end{proof}

We end this section with a remark on the power series ring.

\begin{remark}\label{Rem 6.7} Let $R$ be an integral domain and $X$ an indeterminate over $R$. If $R[\![X]\!]$ is a $t$-B\'ezout $t$-SP-domain, then $R$ is a $t$-B\'ezout $t$-SP-domain.
\end{remark}

\begin{proof} Let $R[\![X]\!]$ be a $t$-B\'ezout $t$-SP-domain. By Corollary~\ref{Cor 4.5} we have to show that the radical of every principal ideal of $R$ is principal. Let $x\in R$. By Corollary~\ref{Cor 4.5} there is some $g\in R[\![X]\!]$ such that $\sqrt[{R[\![X]\!]}]{xR[\![X]\!]}=gR[\![X]\!]$. Let $g_0$ be the constant coefficient of $R$. It is sufficient to show that $\sqrt{xR}=g_0R$. We have clearly that $\sqrt{xR}\subseteq\sqrt[{R[\![X]\!]}]{xR[\![X]\!]}=gR[\![X]\!]$. Consequently, if $f\in\sqrt{xR}$, then $f=gh$ for some $h\in R[\![X]\!]$, and hence $f=g_0h_0\in g_0R$. To prove the converse inclusion, observe that $g^k=xy$ for some $k\in\mathbb{N}$ and $y\in R[\![X]\!]$. Therefore, $g_0^k=xy_0\in xR$, and thus $g_0\in\sqrt{xR}$.
\end{proof}

Note that the converse of Remark~\ref{Rem 6.7} is not true, since there is a factorial domain $S$ for which $S[\![X]\!]$ is not factorial (as shown in \cite{S}). Clearly, $S$ is a $t$-B\'ezout $t$-SP-domain and a Krull domain. Therefore, $S[\![X]\!]$ is a Krull domain as well, but it fails to be a $t$-B\'ezout domain, since a $t$-B\'ezout Krull domain is obviously a factorial domain.

\bigskip
\textbf{Acknowledgements.} This work was supported by the Austrian Science Fund FWF, Project Number J4023-N35. We would like to thank the referee for carefully reading the manuscript and for many valuable suggestions and comments which improved the quality of this paper.

\end{document}